\documentclass[akc]{ars_imsart}

\RequirePackage{amsthm,amsmath,amsfonts,amssymb}
\RequirePackage[authoryear]{natbib}%% uncomment this for author-year citations
\usepackage[dvipsnames]{xcolor}
\RequirePackage[colorlinks, citecolor=ForestGreen,urlcolor=RoyalBlue]{hyperref}%% uncomment this for coloring bibliography citations and linked URLs
\RequirePackage{graphicx}%% uncomment this for including figures
\graphicspath{{plots/}}
\usepackage{subfig} % Required for creating figures with multiple parts (subfigures)
\usepackage{amsthm,mathrsfs,mathtools,bm,dsfont} % For including math equations, theorems, symbols, etc
\usepackage[T1]{fontenc} % Use 8-bit encoding that has 256 glyphs
\usepackage[utf8]{inputenc}
\usepackage[english]{babel}
\usepackage{xparse}
\usepackage{etoolbox}
\usepackage[shortlabels]{enumitem}
\usepackage{letltxmacro}

\startlocaldefs
\theoremstyle{plain}
\newtheorem{theorem}{Theorem}
\newtheorem{lemma}{Lemma}
\newtheorem{corollary}{Corollary}
\newtheorem{proposition}{Proposition}

\theoremstyle{remark}

\newtheorem{definition}{Definition}

\newcommand{\N}{\mathbb{N}}

\newcommand{\R}{\mathbb{R}}

\newcommand{\bigO}[1]{\mathcal{O}\left( #1 \right)}							% big O				\bigO{...}

\newcommand \dd[1]  { \,\textrm d{#1}}   % infintesimal						% differential			\dd{...}
\newcommand*{\dpar}{\mathop{}\!\partial}									% partial derivative
\newcommand{\E}[1]{\mathbb{E} \left[ #1 \right]}							% expectation			\E{...}
\newcommand{\e}{\textrm{e}}											% exponential			\e
\newcommand{\indic}{\mathds{1}}										% indicator			\indic
\newcommand{\Pro}[1]{\mathbb{P}\! \left( #1 \right)}							% probability			\Pro{...}

\newcommand{\relmiddle}[1]{\mathrel{}\middle#1\mathrel{}}					% conditionals			\con
\newcommand{\con}{\relmiddle\vert}

\newcommand{\smallo}[1]{o\left( #1 \right)}								% small o				\smallo{...}
\newcommand{\calL}{\mathcal{L}}

\newcommand{\bzero}{\boldsymbol{0}}

\newcommand\ind{\protect\mathpalette{\protect\independenT}{\perp}}
\def\independenT#1#2{\mathrel{\rlap{$#1#2$}\mkern2mu{#1#2}}}			% independence     			\ind
\endlocaldefs
\newcommand{\omr}{1-\rho^2}

\newcommand{\sqln}{\sqrt{\frac{2(1-\rho^2)}{\log n}}}

\newcommand{\lln}{\log\left(4\pi\log n\right)}

\newcommand{\llln}{\frac{\log\left(4\pi\log n\right)}{\log n}}

\newcommand{\tlan}{\tilde{a}_n}
\newcommand{\tlbn}{\tilde{b}_n}
\newcommand{\nto}{\underset{n\to\infty}{\to}}
\newcommand{\simn}{\underset{n\to\infty}{\sim}}

\begin{document}

\begin{frontmatter}
\title{On the relation between extremal dependence and concomitants}

\begin{aug}

\author[A]{\fnms{Amir}~\snm{Khorrami Chokami}\ead[label=e1]{amir.khorramichokami@unito.it}},
\author[B]{\fnms{Marie}~\snm{Kratz}\ead[label=e2]{kratz@essec.edu}}
% \and
% \author[B]{\fnms{???}~\snm{???}\ead[label=e3]{???@???}}
%%%%%%%%%%%%%%%%%%%%%%%%%%%%%%%%%%%%%%%%%%%%%%
%% Addresses                                %%
%%%%%%%%%%%%%%%%%%%%%%%%%%%%%%%%%%%%%%%%%%%%%%
\address[A]{ESOMAS Department, University of Torino, Torino, Italy\printead[presep={,\ }]{e1}}

\address[B]{ESSEC Business School, CREAR,  France\printead[presep={,\ }]{e2}}
\end{aug}

\begin{abstract}
    The study of concomitants has recently met a renewed interest due to its applications in selection procedures. For instance, concomitants are used in ranked-set sampling, to achieve efficiency and reduce cost when compared to the simple random sampling. 
    In parallel, the search for new methods to provide a rich description of extremal dependence among multiple time series has rapidly grown, due also to its numerous practical implications and the lack of suitable models to assess it. 
    Here, our aim is to investigate extremal dependence when choosing the concomitants approach. In this study, we show how the extremal dependence of a vector $(X, Y)$ impacts the asymptotic behavior of the maxima over subsets of concomitants. Furthermore, discussing the various conditions and results, we investigate how transformations of the marginal distributions of $X$ and $Y$ influence the degeneracy of the limit.
\end{abstract}



\end{frontmatter}

%%%%%%%%%%%%%%%%%%%%%%%%%%%%%%%%%%%%%%%%%%%%%%
%\tableofcontents

%%%%%%%%%%%%%%%%%%%%%%%%%%%%%%%%%%%%%%%%%%%%%%
%%%% Main text entry area:
\section{Introduction} \label{sec:intro}
%%%%%%%%%%%%%%%%
Let $(X_i,Y_i),\,i=1,\dots,n$, be a sequence of independent and identically distributed (i.i.d.) bivariate random vectors with common cumulative distribution (cdf) function $F(x,y)$. We define the \emph{concomitants of order statistics} as follows. Arrange the $X$ variates in ascending order, thus obtaining the order statistics $X_{(1)}\leq X_{(2)}\leq\dots\leq X_{(n)}$ for the $X$ variable. Then, the $Y$ variable corresponding to the $i$-th order statistic $X_{(i)}$ is called the concomitant of the $i$-th order statistic, and is denoted by $Y_{[i]}$. For example, if the maximum of the $X$ variates is $X_{(n)}\equiv X_3$, then its concomitant is $Y_{[n]}\equiv Y_3$.  
Concomitants of order statistics are of notable interest in practical applications: It is worth mentioning their utility in problems concerning the estimation of parameters for multivariate datasets affected by type II censoring, in the development of selection procedures (where it is more convenient to measure a characteristic linked to a measure of interest), in the setting of ranked-set sampling schemes. In particular, the latter deals with situations where the measurement of the variable of interest is time-consuming or expensive, but the ranking of a set of observations linked to the target variable is easy to be performed. For such cases, ranked-set sampling is a better alternative with respect to the simple random sampling, both in terms of efficiency and reduced cost, hence this renewed  attention on the study of concomitants. We report the seminal works \cite{DellClutter1972}, \cite{StokesSager1988}, and more recently \cite{WangStokes2006}, \cite{Shao2010}, \cite{Balci2013}, \cite{Wang2016}, \cite{Wang2017}, \cite{Zamanzade2018}, \cite{Ozturk2019} among others.

Our motivation is to investigate extremal dependence when choosing the concomitant approach. Looking at the literature, this topic has been tackled in the 90's. One can mention the seminal paper \cite{Nagaraja1994} where, given a number $k$, they study the distribution of the maximum of the concomitants $V_1=\max\left( Y_{[n]},\dots,Y_{[n-k+1]}\right)$.
Finite and asymptotic results are given in both cases $k$ a fixed integer (\emph{extreme case}) and $k=[np],\,p\in (0,1)$ (\emph{quantile case}). This leads to the natural question of how much $V_1$ is close to the maximum $Y_{(n)}$, a fundamental problem in practical applications, tackled in \cite{Joshi1995}. In their work, authors propose to link this question to the study of the joint distribution of two maxima of concomitants. More precisely,
given an i.i.d.~sequence of bivariate random vectors with parent random vector $(X,Y)$, define, for $k\geq 1$,
\begin{equation}\label{eq: maxima of conc}
    V_1=\max\left( Y_{[n]},\dots,Y_{[n-k+1]}\right)\quad \text{and} \quad V_2=\max\left( Y_{[n-k]},\dots,Y_{[1]}\right).
\end{equation}
\cite{Joshi1995} introduced the random variable~$W_k:=V_1/Y_{(n)}$ and proved that its cdf~satisfies 
\[
\Pro{W_k\leq w} = \Pro{V_1\leq wV_2,\,V_2>0},
\]
for which they need to compute the joint distribution of $(V_1,V_2)$.
%
%%%%%%%%%%%%%%%%%%

Before briefly reporting their result, let us introduce some notations that will be used throughout the paper.
\\[1ex]
Let $G(z; \mu,\sigma,\xi)$, for $\mu,\xi\in\R,\sigma>0$, denote a Generalised Extreme Value distribution (GEV). A distribution
function $F$ is said to belong to the \emph{Maximum Domain of Attraction} (MDA) of $G$, written as $F\in\mathscr{D}(G)$, if there exist normalizing sequences of constants $a_n>0$ and $b_n\in\R$ such that
$F^n(a_n z + b_n) \to G(z)$, as $n\to\infty$ for any $z\in\Gamma_G=\{z\colon \xi(z - \mu)/\sigma>0\}$. The structure of a MDA is studied in \cite{FisherTippett1928}, \cite{Gnedenko1943}, among others, and more recently, with another view, in \cite{Leonetti2022}. Sufficient conditions for a distribution $F$ to belong to the MDA of a certain GEV $G$ are the so called Von Mises conditions: We refer to \cite[Propositions 1.15 - 1.17]{Resnick1987} for a thorough description and provide a brief recall in Appendix~\ref{app:def} for the paper to be self-contained.
\\[1ex]
A measurable function $h\colon\R_+ \mapsto \R_+$ is regularly varying at infinity with index $a$ (written $h\in\text{RV}_a$) if for $x>0$
\[
    \lim_{t\to\infty}\frac{h(tx)}{h(t)}=x^a.
\]
The index $a$ is called the \emph{exponent of variation}. When $a$ is $0$, $h$ is said \emph{slowly varying} (SV).\\[1ex]
The \emph{upper-tail dependence coefficient}
of a random vector $(X,Y)$ with copula $C$ and marginal distributions $F_X$ and $F_Y$, respectively, is defined as
(\cite{Coles_Heffernan1999})
\begin{equation}\label{eq: UTDC_def}
\lambda_u:=\lim_{\alpha\to 1}\Pro{Y>F_Y^{-1}(\alpha)\mid X>F_X^{-1}(\alpha)} = \lim_{\alpha\to 1}\frac{1-2\alpha-C(\alpha,\alpha)}{1-\alpha},
\end{equation}
$F^{-1}$ denoting the generalized inverse function of a cdf $F$.\\
Note that $\lambda_u\in[0,1]$, where $\lambda_u=0$ means extremal independence between $X$ and $Y$.\\[1ex]
%\subsection{State of the art}
Let us also introduce the following notation, borrowed from \cite{Joshi1995}. Given a vector $(X,Y)$, we define
\begin{align}
& F_1(y\mid x):=\Pro{Y\leq y\mid X>x}, \label{def-F1}\\
& F_2(y\mid x):=\Pro{Y\leq y\mid X\leq x}, \label{def-F2}\\
& F_3(y\mid x):=\Pro{Y\leq y\mid X=x}.  \label{def-F3}
\end{align}
The conditional distributions $F_1,\,F_2$ and $F_3$ play a fundamental role in computing the distributions of the concomitants of the order statistics as well as the distributions of their transformations.
\\[1ex]

\textit{Theorem 2 in \cite{Joshi1995}: Suppose $F_X$ satisfies one of the Von Mises conditions and, for all $x$ and $y$, assume $F_X\in\mathscr{D}(G_X)$, $F_Y\in\mathscr{D}(G_Y)$, and 
\begin{equation}\label{eq: JN95_Theorem_2_joint_extreme_indep_XY}
n\Pro{X>a_nx+b_n,Y>c_ny+d_n}\nto 0,\quad\forall x,y\in\R,
\end{equation}
where the norming constants ${(a_n)}_n,\,{(b_n)}_n,\,{(c_n)}_n\text{ and }{(d_n)}_n$ are such that 
\begin{equation}\label{eq: F3 to 1}
    F_3(c_ny+d_n\mid a_nx+b_n)\to 1 \qquad \text{ as }n\to\infty.
\end{equation}
Further, suppose there exist constants $\tilde{a}_n>0$ and $\tilde{b}_n\in\R$ such that
\begin{equation}\label{eq: JN95_Theorem_2_F1_limit}
F_1(\tilde{a}_n y+\tilde{b}_n \mid a_n x+b_n)\underset{n\to\infty}{\to} H_1(x,y).
\end{equation}
Then, the cdf~of $(V_1,V_2)$ satisfies
\begin{equation}\label{eq: JN95_Theorem_2_joint_V1V2}
F_{(V_1,V_2)}(\tilde{a}_n v_1+\tilde{b}_n,\,c_n v_2+d_n)\nto H(v_1)G_Y(v_2),\quad\forall v_1,v_2\in\R,
\end{equation}
where
$H(v_1)={(k!)}^{-1}\int H_1^k(x,v_1){\left(-\log G_X(x)\right)}^kg_X(x)\dd{x}$.
}\\[1ex]

\noindent We begin our analysis by focusing on Condition \eqref{eq: F3 to 1}. We first show in Lemma~\ref{lemma: nP to 0 equiv F_3 to 1} the relation between \eqref{eq: JN95_Theorem_2_joint_extreme_indep_XY} and \eqref{eq: F3 to 1}, which proof is given in Appendix~\ref{app: proofs}.

\begin{lemma}\label{lemma: nP to 0 equiv F_3 to 1}%\marginpar{\AKC{Our guess}}
Let $x,y\in\R$. Then, we have
\begin{align}\label{eq: equivalence_JN_conditions}
%\text{\eqref{eq: JN95_Theorem_2_joint_extreme_indep_XY}} 
& n\Pro{X>a_nx+b_n,Y>c_ny+d_n} \underset{n\to\infty}{\to} 0  \nonumber\\
&\Leftrightarrow\quad F_3\left(c_ny+d_n\mid a_nz+b_n\right) \underset{n\to\infty}{\to} 1,\;\forall z\geq x.
\end{align}
\end{lemma}
Now, we can see that \eqref{eq: F3 to 1} can be directly understood in terms of the upper-tail dependence coefficient, namely:
\begin{proposition}\label{prop: lambda_U=0 - F_3 to 1}
Consider a bivariate random vector $(X,Y)$ with copula $C$. Then,
\[
F_3\left(c_n y +d_n\mid a_n x + b_n\right) \underset{n\to\infty}{\to} 1,\quad\forall x\geq y\quad \Leftrightarrow\quad \lambda_u=0
\]
where the normalizing constants are chosen depending on the MDA to which $X$ and $Y$ belong.
\end{proposition}
Note that Proposition~\ref{prop: lambda_U=0 - F_3 to 1} is then simply another way to express Proposition 5.27 in \cite[p. 296]{Resnick1987}, which states that 
\[
\lambda_u=0\quad \Leftrightarrow \quad n\Pro{X>a_n x + b_n,Y>c_n y +d_n}   \underset{n\to\infty}{\to} 0.
\]
Using Lemma~\ref{lemma: nP to 0 equiv F_3 to 1} and Proposition~\ref{prop: lambda_U=0 - F_3 to 1}, we can now rewrite Theorem 2 in \cite{Joshi1995}  in terms of $\lambda_u$.
This reformulation, given in Corollary~\ref{coro: JN95_Theorem_2}, is helpful since it clarifies the connection between the asymptotic independence of $(V_1,V_2)$ and that of $(X,Y)$. 
\begin{corollary}\label{coro: JN95_Theorem_2}
    Suppose $F_X$ satisfies one of the Von Mises conditions and, for all $x,\,y$, $F_X\in\mathscr{D}(G_X)$ and $F_Y\in\mathscr{D}(G_Y)$. Moreover, assume that the upper tail dependence coefficient between $X$ and $Y$ is $0$: $\lambda_u=0$.
If there exist constants $\tilde{a}_n>0,\,\tilde{b}_n$ s.t. $F_1(\tilde{a}_n y+\tilde{b}_n \mid a_n x+b_n)\underset{n\to\infty}\to H(x,y)$, then the cdf $F_{(V_1,V_2)}$ of $(V_1,V_2)$ satisfies \eqref{eq: JN95_Theorem_2_joint_V1V2}.
\end{corollary}

The assumption of asymptotic independence between $X$ and $Y$ is generally too restrictive in practical applications, especially in this highly interconnected world. For instance, over the past 15 years, we have had to live through a major global financial crisis, then a pandemic, combined with an increase in cyber attacks. The presence of systemic risk, with a strong dependence among extremes, is one characteristic of all those events (\emph{e.g.}~\cite{Dacorogna2015}). Hence, investigating this (asymptotic) dependence aspect is of primary interest.  This is the main objective of the next section.

%----------------------------------------------------------------------------------------
%	MAIN RESULTS
%----------------------------------------------------------------------------------------

%%%%%%%%%%%%%%%%%%%%%%%%%%%%%%%%%%%%%%%%%%%%%%
%% Single Appendix:                         %%
%%%%%%%%%%%%%%%%%%%%%%%%%%%%%%%%%%%%%%%%%%%%%%
%\begin{appendix}
\section{Joint asymptotics for maxima over subsets of concomitants} \label{sec:mainResult}

We aim at studying the limiting distribution of the concomitants $(V_1,V_2)$ when $X$ and $Y$ may exhibit some asymptotic dependence. To do so, we take advantage of the joint-tail model representation given in \cite{Ledford1998}, then we state the main result (in Subsection \ref{ssec:theo}) and discuss the choice of the normalizing constants to avoid degeneracy in the limit.

\vspace{-2ex}
%%%%%%%%%%%%%%%%%%%%%%%%%%
\subsection{The LT joint-tail model}
\label{ssec:LTmodel}
%%%%%%%%%%%%%%%%%%%%%%%%%%
%
First, let us recall some definitions related to the slowly varying notion in the bivariate case. It will be needed to define the joint-tail model.
\begin{definition}
    A function $\calL$ that satisfies
    \begin{equation}\label{eq: BSV_definition}
        \lim_{n\to\infty} \frac{\calL(nx,ny)}{\calL(n,n)}=r(x,y),\qquad\text{with }r(ax,ay) = r(x,y),
    \end{equation}
    for all $a>0$ and $(x,y)\in\R^2_+$, is said to be \emph{bivariate slowly varying} (BSV).
\end{definition}
It can be shown that there exists a univariate function $r_*$ such that 
\begin{equation}\label{eq:rstar}
r(x,y)=r_*(w), \quad\mbox{with}\quad w=\frac{x}{x+y}. 
\end{equation}
\begin{definition}\label{def: BSV}
    A BSV function $\calL$ for which the function $r_*$ defined in \eqref{eq:rstar} satisfies
    \begin{equation*}
        \frac{r_*(w)}{r_*(1-w)}\text{ is SV at $w=0$ and $w=1$},
    \end{equation*}
    is said to be \emph{quasi-symmetric}.
\end{definition}
Let us now present the joint-tail model, named the LT  joint-tail model, introduced in \cite{Ledford1998} to study the asymptotic distribution of the concomitant of the max, $Y_{[n]}$, {\textit i.e.}~the limit distribution of the marginal distribution of $V_1$ (case $k=1$), and to compute $\Pro{Y_{(n)}=Y_{[n]}}$.

Let $(X, Y)$ be a bivariate random vector with unit Fréchet marginals (\emph{i.e.} $F_X(x)=\e^{-1 / x}$, $x>0$) and joint distribution function $F(x, y)$. Suppose that there exist functions $c(t)$ and $\psi(x, y)$ for which the joint survival function \[\bar{F}(x, y)=1-\e^{-1 / x}-\e^{-1 / y}+F(x, y)\] satisfies
\begin{equation}\label{cond: LT98_2.3}
\lim_{t \rightarrow \infty} \frac{t \bar{F}(t x, t y)}{c(t)}=\psi(x, y) \text { for all }(x, y) \in \R_{+}^2.
\end{equation}
Additionally suppose that there exists $\gamma \in \R$ such that
\begin{equation}\label{cond: LT98_2.4}
    \begin{cases}
    \bar{F}(0, t)/\bar{F}(t, 0) &\text{ is  } \text{RV}_\gamma \text{ at infinity and}\\
    \bar{F}(t, 0)/\bar{F}(0, t) &\text{ is  } \text{RV}_{-\gamma}\text{ at infinity}.
    \end{cases}
\end{equation}
Then the following holds.
\begin{definition}[The LT joint-tail model]\label{def-LTmodel}
Assume $(X,Y)$ have unit Fréchet marginals. Let $\eta\in (0,1]$ be the \emph{coefficient of tail dependence} used to determine the decay rate of $F ^{-1}(t,t)$, as $t\to\infty$. 
Under Assumptions~\eqref{cond: LT98_2.3} and \eqref{cond: LT98_2.4}, the joint survival distribution of $(X,Y)$ is given by
\begin{equation}\label{eq: Ledford98_joint-tail}
\Pro{X>x,Y>y}=\calL(x,y)x^{-\alpha}y^{-\beta},\quad \forall x,y,
\end{equation}
where $\alpha,\beta>0,\,\alpha+\beta=\eta^{-1}$ and the function $\calL(x,y)$ is a quasi-symmetric BSV function (see Definition~\ref{def: BSV}).
\end{definition}
The coefficient  $\eta$ describes the type of limiting dependence between $X$ and $Y$, while the function $\calL$ its relative strength given a particular value of $\eta$; see \cite{Ledford1998} and \cite{Heffernan2000} for further comments.
As pointed out in \cite{DeHaan_Zhou2011}, the $\eta=1/2$ corresponds to the case where $X$ and $Y$ are independent, while the case $\eta > 1/2$ indicates a positive association of the extremes of $(X,Y)$.
Values of $\eta$ lower than $1/2$ represent cases where the extremes of $X$ and $Y$ are negatively associated.
The bounding cases of perfect negative and positive dependence correspond respectively to $\eta\to 0$ and $\eta=1$ with $\calL(\cdot)=1$.
\cite{Coles_Heffernan1999} introduce an elementary measure of dependence: Since the survival copula $\bar C$ is given by $\bar{C}(u,v)=1-u-v+C(u,v)$, they consider the following indices to describe the upper-tail dependence:
\begin{align}\label{eq: relation_lambda_U_eta}
&\bar{\chi}:=\lim_{u\to 1}\frac{2\log(1-u)}{\log\bar{C}(u,u)}-1=2\eta-1\\
&\lambda_u:=2-\lim_{u\to 1}\frac{\log C(u,u)}{\log u}=
\begin{cases}
c&\text{ if } \bar{\chi}=1,\,\calL(t)\to c>0,\,\text{as }t\to\infty,\\
0&\text{ if } \bar{\chi}=1,\,\calL(t)\to 0,\,\text{as }t\to\infty,\\
0&\text{ if } \bar{\chi}<1.
\end{cases}
\end{align}

Since $\bar{\chi}=1$ if and only if $\eta=1$, we have a link between $\lambda_u$ and $\eta$ as well. 
In other words,
\begin{itemize}[leftmargin=*]
    \item $X$ and $Y$ are asymptotically independent when $\alpha+\beta=1$ and $\calL(n,n)\underset{n\to\infty}{\to} 0$, or when $\alpha+\beta\geq 1$ (with no condition on the limit as $n\to\infty$ of $\calL(n,n)$); 
    \item $X$ and $Y$ are asymptotically dependent when $\alpha+\beta=1$ and $\calL(n,n)\underset{n\to\infty}{\to}  c>0$. Note that, in this case, $c=\lambda_u$.
\end{itemize}
%------------------
This turns out to be useful when considering extensions of Theorem 2 in~\cite{Joshi1995}, relaxing Condition~\eqref{eq: F3 to 1}.

For the sake of simplicity, we will use the notation $C_{\not\ind}$ to indicate the asymptotic dependence case, that is $\alpha+\beta=1$ and $\calL(n,n)\underset{n\to\infty}{\to} \lambda_u>0$. Thus, the indicator $\indic_{C_{\not\ind}} = 1$ points out the presence of asymptotic dependence.
%
%%%%%%%%%%%%%%%%%
\subsection{Main result}\label{ssec:theo}
%%%%%%%%%%%%%%%%%
%
Let us state our main result on the joint asymptotic behaviour of maxima of concomitants, allowing for asymptotic dependence of $X$ and $Y$:
\begin{theorem}%[Joint asymptotic behavior of Maxima of Concomitants]
\label{th: Joint_max_of_conc}
Let $(X,Y)$ be a bivariate random vector with unit Fréchet marginals and upper-tail dependence coefficient $\lambda_u$. Assume that $(X,Y)$ follows the LT joint-tail model given in Definition~\ref{def-LTmodel}, with survival cdf~\eqref{eq: Ledford98_joint-tail}, % We have to be careful that it means also extra regularity cdtions, reason why I changed the presentation of the conditions.
% $\bar{F}(x,y)=\calL(x,y)x^{-\alpha}y^{-\beta},\,\alpha \in (0,1), \beta>0$;
%
and  that the BSV function $\calL$ is such that
\begin{align}\label{cond: limit of L(n,a_n)} 
& \qquad\qquad \qquad \calL(nx,\tilde{a}_ny+\tilde{b}_n)\underset{n\to\infty}{\to} \tilde{c}(x,y)\quad \text{where} \\
& \tilde{c}(x,y)\,x^{1-\alpha}<y^{\beta}, \;\; \tilde{c}(x,y)\underset{y\to 0}{\sim} y^{-\beta},\;\; \tilde{c}(x,y)=\smallo{y^{-\beta}} \;\text{for}\; y\to\infty, \nonumber\\
& \text{with}  \quad 
\tilde{a}_n=\mathcal{O}\left(n^{\frac{1-\alpha}{\beta}}\right)\quad\text{ and }\quad \tilde{b}_n={O}\left(\tilde{a}_n\right) \text{ or }\, o\left(\tilde{a}_n\right). \nonumber
\end{align}
Then, the joint distribution of the concomitants maxima $(V_1,V_2)$ defined in \eqref{eq: maxima of conc} satisfies:
%\begin{equation}\label{eq: JN95_Theorem_2_joint_V1V2}
\begin{equation}\label{eq: Joint_V1V2}
    \begin{split}
        &F_{(V_1,V_2)}(\tilde{a}_n v_1+\tilde{b}_n,\,n v_2)\\
        &\xrightarrow[n\to\infty]{} \int_0^{+\infty} H_1^k(v_1\mid x)\,H_2(v_2\mid x)\,\frac{x^{-k-2}}{k!}\,\emph{\e}^{-1/x}\dd{x},\quad\forall v_1,v_2\in\R ,
    \end{split} 
\end{equation}
where $H_1(y\mid x):=1-\tilde{c}(x,y)\,x^{1-\alpha}\,y^{-\beta}$ and 
$\displaystyle H_2(y\mid x)$ is defined by the product of the limits given below in \eqref{eq: F2_limit} and \eqref{eq: F3_limit}, respectively, with $r(x,y)=\lambda_u^{-1}\tilde{c}(x,y)$.
\end{theorem}

\vspace{2ex}
To prove Theorem~\ref{th: Joint_max_of_conc}, we need to evaluate the asymptotic behavior of the three conditional distributions $F_1$, $F_2$ and $F_3$ defined in \eqref{def-F1}, \eqref{def-F2} and \eqref{def-F3}, respectively, each adequately transformed to obtain a non-degenerated limit distribution. This is what is presented in the following lemma.
\begin{lemma}\label{lemma: F2_F3_limit}
    Assume the LT joint-tail model \eqref{eq: Ledford98_joint-tail} holds and call $r$ the limit function corresponding to $\calL$. Then, as $n\to\infty$,
    \begin{align}
        F^{n}_{2}(ny\mid nx)&\to \exp\left\{-\frac 1{y}\left(1-\lambda_u r(x,y){\left(\frac yx\right)}^{\alpha}\indic_{C_{\not\ind}}\right)\right\},\label{eq: F2_limit}\\
        F_3(ny\mid nx)&\to 1 + \lambda_u x^{-\alpha+2}y^{\alpha-1}\left(\frac{\dpar}{\dpar x}r(x,y)-\frac{\alpha}{x}r(x,y)\right)\indic_{C_{\not\ind}}.\label{eq: F3_limit}
    \end{align}
Moreover, under Condition~\eqref{cond: limit of L(n,a_n)}, we have
    \begin{align}
        F_{1}(\tilde a_n y+ \tilde b_n\mid nx)& \underset{n\to\infty}{\to}\,\tilde{c}(x,y).
    \end{align}
\end{lemma}
%
%\begin{remark}
Note that taking $\indic_{C_{\not\ind}}=0$ in \eqref{eq: F2_limit} and \eqref{eq: F3_limit} provides 
$\displaystyle \lim_{n\to\infty} F^{n}_{2}(ny\mid nx) =\e^{-1/y}$ and $\displaystyle \lim_{n\to\infty} F_3(ny\mid nx)=1$, which characterise copulas with asymptotic independence, as considered in \cite{Joshi1995}.
The proofs of Theorem~\ref{th: Joint_max_of_conc} and Lemma~\ref{lemma: F2_F3_limit} are developed in Appendix~\ref{app: proofs}.

It is worth noticing that Condition~\eqref{cond: limit of L(n,a_n)} is needed to ensure that $H_1(v_1\mid x)$ is a proper cdf~for $v_1\in (0,+\infty)$.
%To better precise this delicate point, 
Moreover, %differently with respect to Theorem 2 in \cite{Joshi1995},
our Theorem~\ref{th: Joint_max_of_conc} gives precise formulations on the choice of the normalizing constants $\tilde{a}_n$ and $\tilde{b}_n$. This is in fact a delicate point, reason why we dedicate a comprehensive discussion on it: Explanations relative to Condition~\eqref{cond: limit of L(n,a_n)} are given in the next subsection, while Section~\ref{sec: Gaussian case} addresses the fundamental example of the Gaussian bivariate distribution, for which Condition~\eqref{cond: limit of L(n,a_n)} does not hold.
%%%%%%%%%%%%%%%%%%%%%%%%%%%%%%%%%%%%%%%%%%%%%%%%%%%%%%%%%%%%%
\subsection{Discussion on the choice of $\tilde{a}_n$ and $\tilde{b}_n$} \label{subsec: discussion_an_bn_tilde}
%%%%%%%%%%%%%%%%%%%%%%%%%%%%%%%%%%%%%%%%%%%%%%%%%%%%%%%%%%%%%
%
In the following, we carefully deduce which are the only possible asymptotic behaviours of the constants $\tlan$ and $\tlbn$. We highlight the importance of the present subsection, as finding suitable constants which lead to non degenerate limits represents generally an issue in asymptotic theories.

We start by noticing that the limit of $F_1(\tilde{a}_n y+\tilde{b}_n\mid nx)$ strongly depends on that of $\calL(nx,\tilde{a}_n y+\tilde{b}_n)$. Indeed, we can write 
    \begin{align}\label{approxF1-L}
        F_1(\tilde{a}_n y+\tilde{b}_n\mid nx)&=\, 1-\frac{\calL(nx,\tilde{a}_n y+\tilde{b}_n){(nx)}^{-\alpha}{(\tilde{a}_n y+\tilde{b}_n)}^{-\beta}}{1-\e^{-1/(nx)}} \nonumber\\
        &\underset{n\to\infty}\sim\, 1-\frac{\calL(nx,\tilde{a}_n y+\tilde{b}_n)}{n^{\alpha-1} x^{\alpha-1}{(\tilde{a}_n y+\tilde{b}_n)}^{\beta}}
    \end{align}
using the approximation $1-\e^{-1/(nx)}\underset{n\to\infty}\sim{(nx)}^{-1}$.

In the following, we discuss which sequences $\tilde{a}_n$ and $\tilde{b}_n$ should be chosen to have $F_1(\tilde{a}_n y+\tilde{b}_n\mid nx)\underset{n\to\infty}{\to} 1$, \emph{i.e.}~$V_1$ not degenerate.

Given a couple $(x,y)$, we study the limit of \eqref{approxF1-L} depending on the behavior of $\calL$, namely:
\begin{enumerate}[leftmargin=*, label=(\roman*)]
\item \label{case: (i)} Assume $\calL(nx,\tilde{a}_n y+\tilde{b}_n)\to 0$ as $n\to\infty$.
To have $F_1(\tilde{a}_n y+\tilde{b}_n\mid nx)\not\to 1$, we need $\frac{\calL(nx,\tilde{a}_n y+\tilde{b}_n)}{n^{\alpha-1} x^{\alpha-1}{(\tilde{a}_n y+\tilde{b}_n)}^{\beta}}\not\to 0$. But, since we are in the case $\calL(nx,\tilde{a}_n y+\tilde{b}_n)\underset{n\to\infty}{\to} 0$, this would mean, either
\begin{itemize}
\item $n^{\alpha-1} x^{\alpha-1}{(\tilde{a}_n y+\tilde{b}_n)}^{\beta}\underset{n\to\infty}{\to} 0$ at the same rate of $\calL$, which is impossible because of the assumptions of the model ($\calL\sim\text{SV}$), or
\item $n^{\alpha-1} x^{\alpha-1}{(\tilde{a}_n y+\tilde{b}_n)}^{\beta}\underset{n\to\infty}{\to} 0$ faster than $\calL$, which is impossible because we would have $\frac{\calL(nx,\tilde{a}_n y+\tilde{b}_n)}{n^{\alpha-1} x^{\alpha-1}{(\tilde{a}_n y+\tilde{b}_n)}^{\beta}}\underset{n\to\infty}{\to} \infty$, that denies $F_1$ is a cdf.
\end{itemize}
Hence case~\ref{case: (i)} is not possible.
\item\label{case: (ii)} Assume $\calL(nx,\tilde{a}_n y+\tilde{b}_n)\to \infty$ as $n\to\infty$. So, as $n\to\infty$,
    the term $n^{\alpha-1} x^{\alpha-1}{(\tilde{a}_n y+\tilde{b}_n)}^{\beta}$ tends to infinity faster than $\calL(nx,\tilde{a}_n y+\tilde{b}_n)$ does, which yields that $F_1(\tilde{a}_n y+\tilde{b}_n\mid nx)\underset{n\to\infty}{\to} 1$.  Hence, this case~\ref{case: (ii)} has also to be discarded.
\item The previous reasoning implies that we can only have, as $n\to\infty$,
\[
    \calL(nx,\tilde{a}_n y+\tilde{b}_n)\underset{n\to\infty}{\to}\tilde{c}(x,y)\in(0,\infty).
\]
Which conditions on $\tilde{a}_n$ and $\tilde{b}_n$ can be given? Here, we have to distinguish whether $\tilde{c}(x,y)$ depends on $y$ or not.
\begin{enumerate}[label=(\Alph*)]
\item \emph{Suppose $\tilde{c}(x,y)$ depends on $y$.} Then,\\ 
$\displaystyle \frac{\calL(nx,\tilde{a}_n y+\tilde{b}_n)}{n^{\alpha-1} x^{\alpha-1}{(\tilde{a}_n y+\tilde{b}_n)}^{\beta}}\underset{n\to\infty}{\sim}\frac{\tilde{c}(x,y)}{n^{\alpha-1} x^{\alpha-1}{(\tilde{a}_n y+\tilde{b}_n)}^{\beta}}$, so we have to consider the relative behaviour of $\tilde{a}_n$ and $\tilde{b}_n$.
    \begin{enumerate}[label=(\alph*)]
    \item If $\tilde{a}_n=o(\tilde{b}_{n})$, then $\tilde{a}_n y+\tilde{b}_n\,{\sim}\,\tilde{b}_{n}$, as $n\to\infty$. This implies\\ 
$\frac{\tilde{c}(x,y)}{n^{\alpha-1} x^{\alpha-1}{(\tilde{a}_n y+\tilde{b}_n)}^{\beta}}\underset{n\to\infty}{\sim}\frac{\tilde{c}(x,y)}{n^{\alpha-1} x^{\alpha-1}{\tilde{b}_n}^{\beta}}$, so, in order to maintain the dependence on $y$, we need $n^{\alpha-1}\tilde{b}_n^{\beta}\underset{n\to\infty}{\not\to} 0,\infty$.
This condition obviously depends on the value of $\alpha$: Call $\gamma=(1-\alpha)/\beta$, then
        \begin{itemize}
        \item If $0<\alpha<1$, $n^{\alpha-1}\underset{n\to\infty}{\to}0$, the only possibility is $\tilde{b}_n=\bigO{n^{\gamma}}$ with $\gamma\in(0,1)$, since $\alpha+\beta\geq 1$.
        \item If $\alpha=1$, then $\tilde{b}_n$ would be independent of $n$; it is then impossible to propose $(\tilde{a}_n,\tilde{b}_n)$ in such a case.
        \item If $\alpha>1$, then $\tilde{b}_n=\bigO{n^{\gamma}}$ with $\gamma<0$. However, since we are interested in the dependence in the extremes, we need to consider
\[
n\Pro{X>nx,Y>n^{\gamma}y}\leq n^{1-\gamma}\cdot n^{\gamma}\Pro{X>n^{\gamma}x,Y>n^{\gamma}y}\underset{n\to\infty}{\not\to} 0,
\]
which implies $\gamma\in(0,1)$. So, the case $\alpha>1$ cannot be considered.
        \end{itemize}
    \item If $\tilde{b}_n=\smallo{\tilde{a}_{n}}$ or $\tilde{b}_n=\bigO{\tilde{a}_{n}}$, then $\tilde{a}_n y+\tilde{b}_n \underset{n\to\infty}{\sim}\tilde{a}_{n}y$\, and
\[
\frac{\tilde{c}(x,y)}{n^{\alpha-1} x^{\alpha-1}{(\tilde{a}_n y+\tilde{b}_n)}^{\beta}}
\underset{n\to\infty}{\sim} \frac{1}{n^{\alpha-1} {\tilde{a}_n}^{\beta}}\,\tilde{c}(x,y)\,x^{1-\alpha}.
\]
Hence, we need to consider $\tilde{a}_n=\bigO{n^{\gamma}}$ and $\alpha\in(0,1)$.
    \end{enumerate}
\item \emph{Suppose $\tilde{c}(x,y)$ does not depend on $y$.}
In this case, 
\[
\frac{\calL(nx,\tilde{a}_n y+\tilde{b}_n)}{n^{\alpha-1} x^{\alpha-1}{(\tilde{a}_n y+\tilde{b}_n)}^{\beta}} \underset{n\to\infty}{\sim} \frac{n^{1-\alpha}}{\max\left(\tilde{a}_n^{\beta} y^{\beta},\tilde{b}_n^{\beta}\right)}\tilde{c}x^{1-\alpha},
\]
and since we do not want to loose the dependence on $y$, we need to assume $\max\left(\tilde{a}_n^{\beta} y^{\beta},\tilde{b}_n^{\beta}\right)=\tilde{a}_n^{\beta} y^{\beta}$. Moreover, we need to ask $n^{1-\alpha}\tilde{a}_n^{-\beta}\not\to 0$ (otherwise $F_1(\tilde{a}_n y+\tilde{b}_n\mid nx)\underset{n\to\infty}{\to}1$), which implies $\tilde{a}_n=\bigO{n^{\gamma}}$ and $\alpha\in(0,1)$.
\end{enumerate}
\end{enumerate}
Therefore, to obtain a non degenerate limit for $F_1(\tilde{a}_n y+\tilde{b}_n\mid nx)$, we need to assume $\alpha\in(0,1)$, $\tilde{a}_n=\bigO{n^{\gamma}}$ and $\tilde{b}_n=\smallo{\tilde{a}_{n}}$ or $\tilde{b}_n=\bigO{\tilde{a}_{n}}$.
Now that the main result is provided with a justification of the constants involved, let us illustrate the theoretical results on standard examples of the risk literature. 
\section{Illustration of the results}\label{sec: Examples}
In this section, we consider two examples: the Pareto - Lomax with Survival Clayton copula and the Bivariate Symmetric Logistic Extremal copula.
We apply our asymptotic approximations given in Theorem~\ref{th: Joint_max_of_conc} and compare it with the empirical results obtained via simulation.
The choice of such examples lies in the fact that the extremal dependence is ruled by parameters, so easy to control. Note also that for the following examples we can achieve the same results by direct computation (with a considerable amount of time and tedious calculations), thus it is possible to make comparisons by means of the mathematical expressions obtained.

We simulate $5000$ replicates of $100000$ bivariate samples. The dimension $k$ of the subsample used to compute $V_1$ is set to $10$ (Appendix~\ref{app: other k} contains simulations relative to other values of $k$, to have insights of its impact on the asymptotic result).

\subsection{Pareto - Lomax marginals with Survival Clayton copula}\label{subsec: counterex_JN95_Pareto - Lomax_SClayton}
Let $(X_P,Y_P)$ be a random vector with Survival Clayton$(\theta)$ copula and Pareto - Lomax$(\nu,1)$ marginals (the subscript $P$ indicates the Pareto marginals):
\begin{align*}
    &\bar{F}_{X_P}(x)={(1+x)}^{-\nu},\,\forall x>0,\quad \bar{F}_{Y_P}(y)={(1+y)}^{-\nu},\,\forall y>0,\quad \nu>0,\\
    &\bar{F}_{P}(x,y)=\Pro{X_P>x,Y_P>y}={\left({(1+x)}^{\nu\theta}+{(1+y)}^{\nu\theta}-1\right)}^{-1/\theta},\,\theta>0.
\end{align*}
\begin{figure}[t]
    \centering
    \subfloat[][\emph{Survival Clayton copula $(\theta=2)$}]{\includegraphics[trim={5cm 1cm 4cm 2cm}, clip, width = 0.45\linewidth]{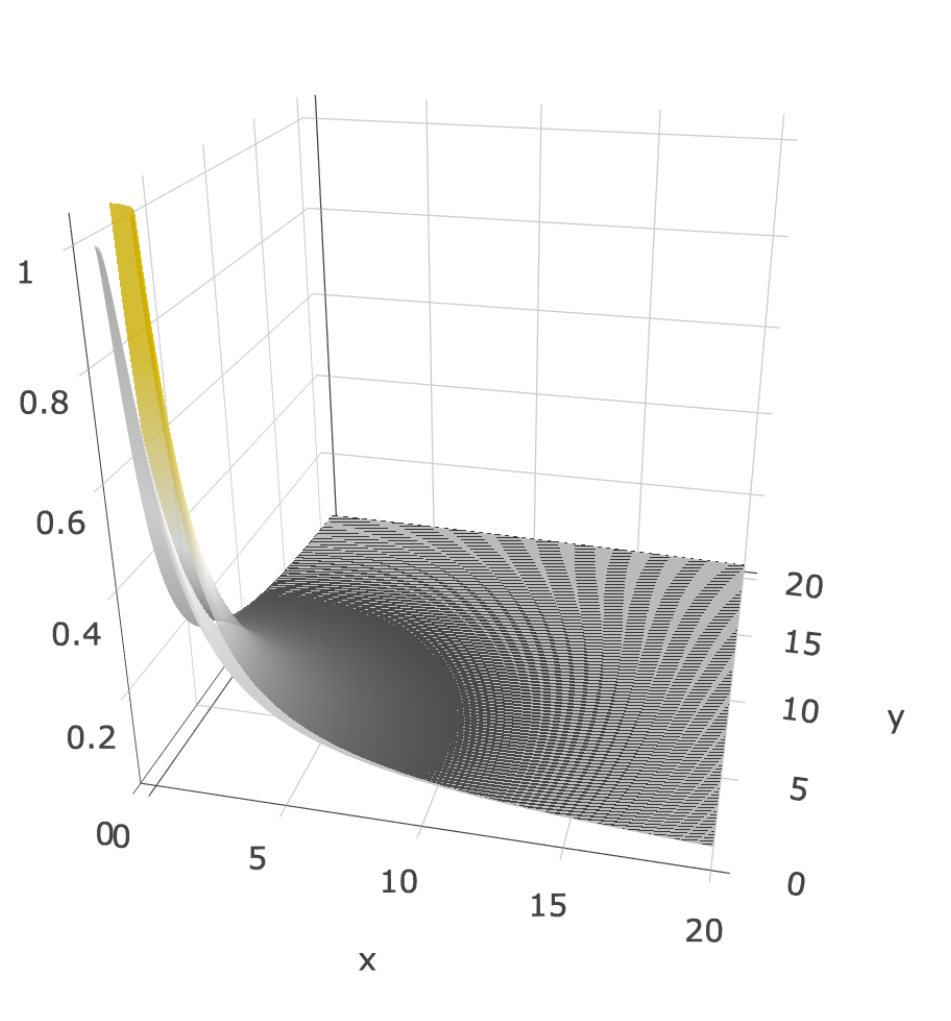}}\quad
    \subfloat[][\emph{Logistic copula $(\gamma=1/2)$}]{\includegraphics[trim={5cm 1cm 4cm 2cm}, clip, width = 0.45\linewidth]{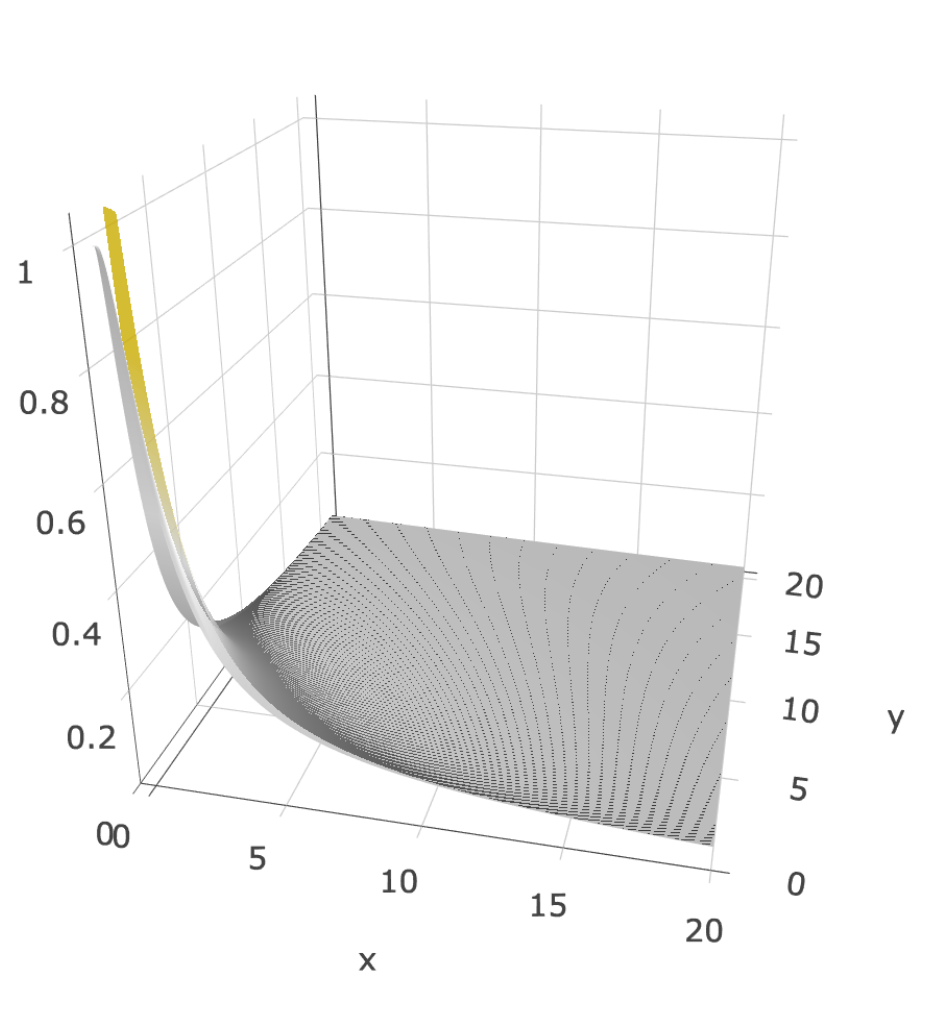}}
    %\subfloat[][\emph{Scatterplot of $(V_1, V_2)$, $(\gamma=2)$}]{\includegraphics[width = 0.45\linewidth]{SC_V1_V2_scatterplots_k10}}
    %
    \caption{Joint tail $\bar{F}(x,y)$ (gray) with unit Fréchet marginals, as expressed in \eqref{eq: joint_SClayton}, and its approximation (black and gold) given by Equation \eqref{eq: Ledford98_joint-tail}. The parameters $\alpha$ and $\beta$ are both set to $0.5$ in the two examples.}\label{fig: Joint tail approximation}
\end{figure}

First, we transform te Pareto marginals into the unit Fréchet ones, to match the assumptions of the model given in Theorem~\ref{th: Joint_max_of_conc}. We can write
\[
F(x,y)=F_{P}\left(F_{X_P}^{-1}\left(\e^{-1/x}\right),F_{Y_P}^{-1}\left(\e^{-1/y}\right)\right),
\]
where, for the Pareto distribution, $F_{X_P}^{-1}(q)=F_{Y_P}^{-1}(q)={(1-q)}^{-1/\nu}-1$. We obtain
\begin{equation}\label{eq: joint_SClayton}
    \bar{F}(x,y)={\left({\left(1-\e^{-1/x}\right)}^{-\theta}+{\left(1-\e^{-1/y}\right)}^{-\theta}-1\right)}^{-1/\theta}.
\end{equation}
Now, observe that the joint survival \eqref{eq: joint_SClayton} can be rewritten, for $x>x_0>0$ and $y>y_0>0$, as
\[
    \begin{split}
        \bar{F}(x,y)&\sim
        {\left(x^{\theta}+y^{\theta}-1\right)}^{-1/\theta}
        =
        {\left(x^{\theta}+y^{\theta}-1\right)}^{-1/\theta}{(xy)}^{1/2}{(xy)}^{-1/2}\\&=\mathcal{L}(x,y)x^{-1/2}y^{-1/2},
    \end{split}
\]
where $\mathcal{L}(x,y)={\left(x^{\theta}+y^{\theta}-1\right)}^{-1/\theta}{(xy)}^{1/2}{(xy)}$ and clearly $\alpha=\beta=1/2$.

Note that
\[
\mathcal{L}(n,n)={(2n^{\theta}-1)}^{-1/\theta}n\to 2^{-1/\theta}\quad\text{as }n\to\infty.
\]
This is coherent with the value of the upper-tail dependence coefficient, since
\[
\begin{split}
\lambda_U=\lim_{q\to 1}\Pro{Y>F_Y^{-1}(q)\mid F_X^{-1}(q)}=\lim_{q\to 1}{(1-q)}^{-1}{(2{(1-q)}^{-\theta}-1)}^{-1/\theta}=2^{-1/\theta}.
\end{split}
\]
\begin{figure}[t]
    \centering
    \subfloat[][\emph{Bivariate distribution of $(V_1,V_2)$.}\label{subfig: SC_comparison_surface}]{\includegraphics[trim={5cm 1cm 4cm 2.5cm}, clip, width = 0.37\linewidth]{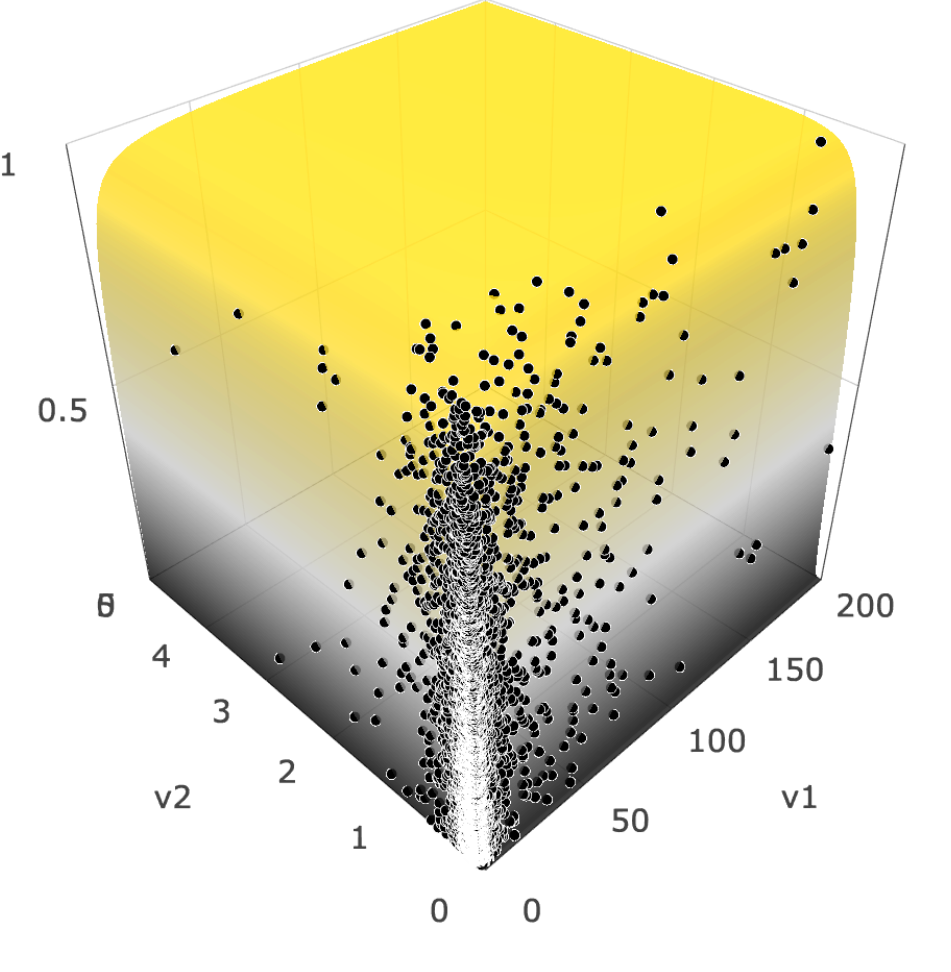}}
    \quad
    \subfloat[][\emph{Error evaluation.}\label{subfig: SC_comparison_boxplot}]{\includegraphics[trim={5cm 0cm 5cm 0cm}, clip, width = 0.168\linewidth]{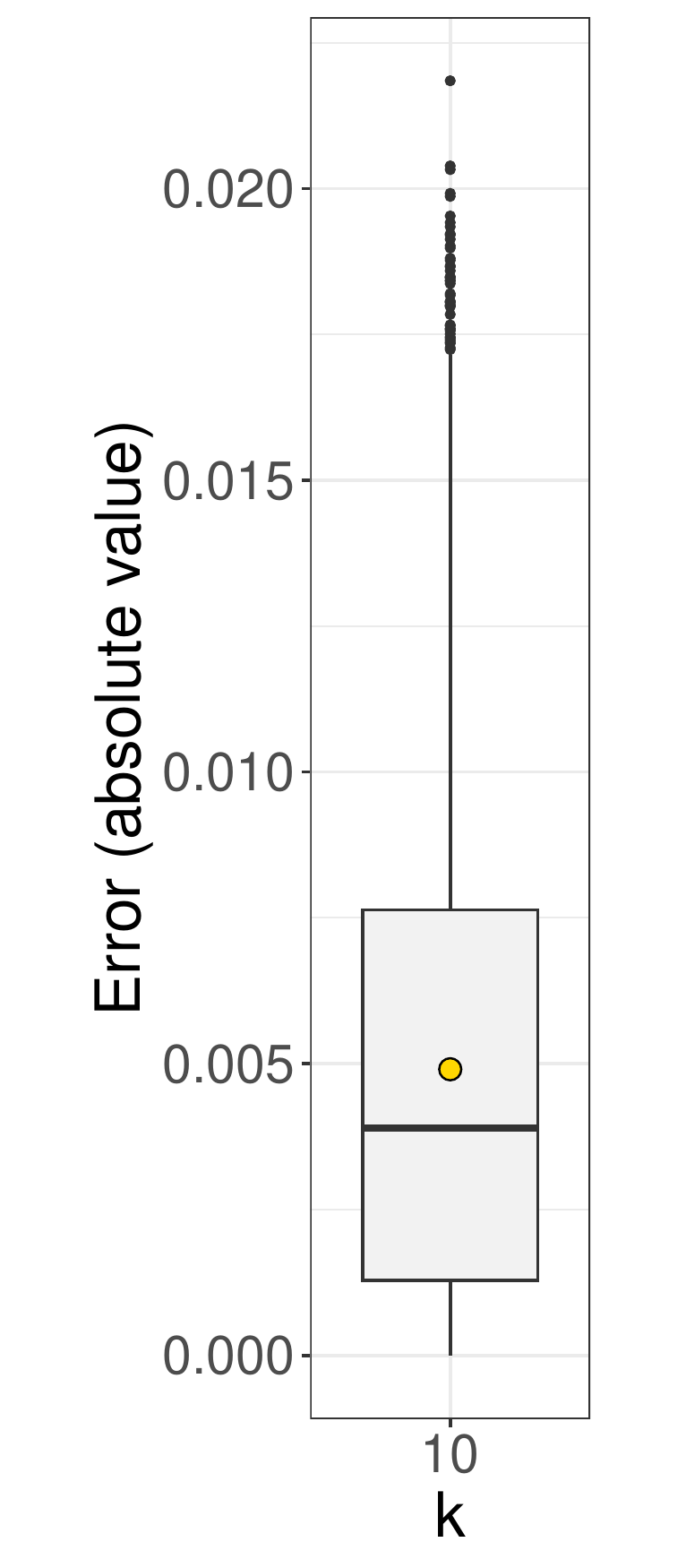}}
    \quad
    \subfloat[][\emph{Marginal distributions diagnostics.}\label{subfig: SC_comparison_margins}]{\includegraphics[width = 0.39\linewidth]{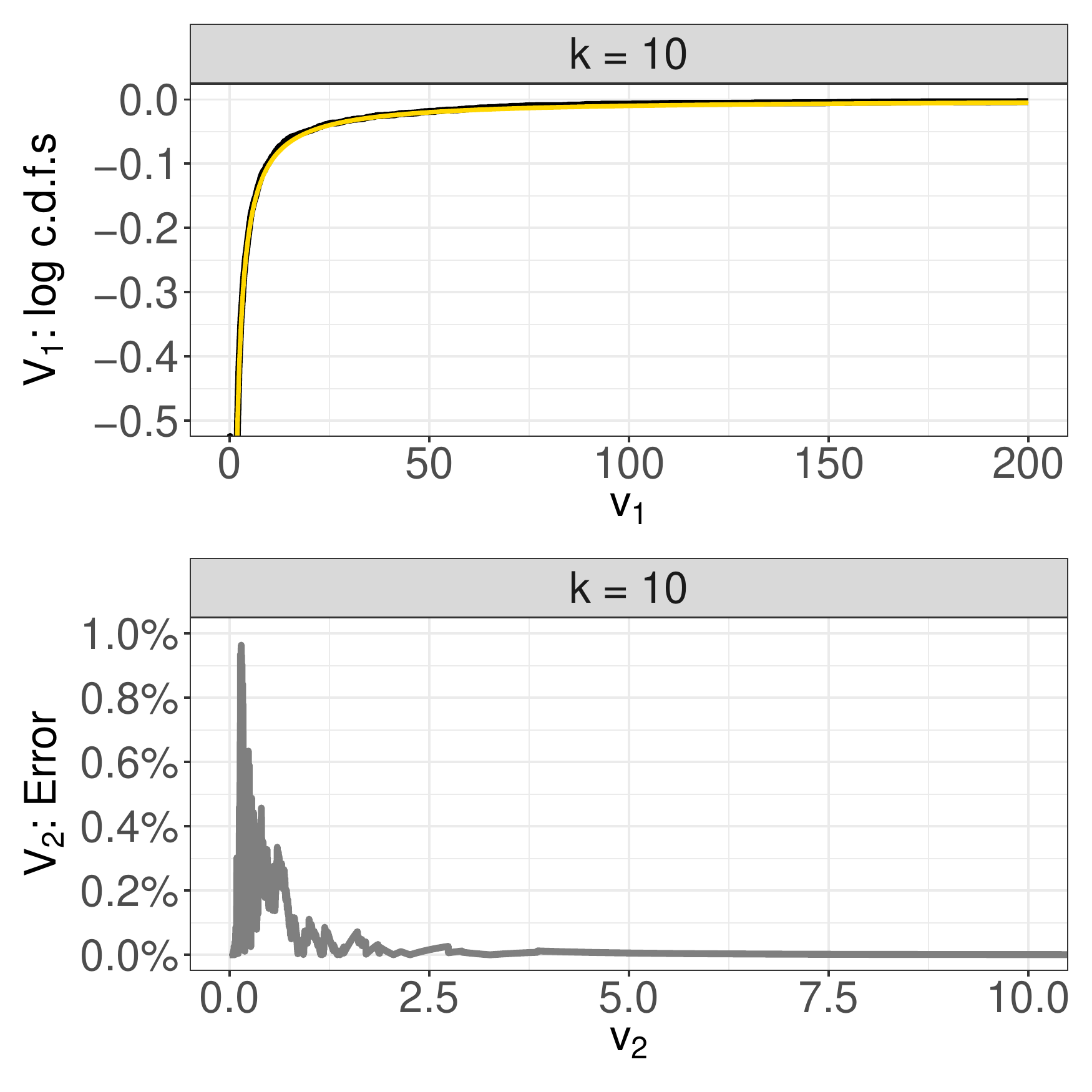}}
    \caption{Survival Clayton$(\theta=2)$, with $k = 10$. Panel (a) shows the empirical univariate log cdf of $V_1$ (in black), together with the asymptotic theoretical log values (in gold) obtained by numerically integrating the asymptotic joint cdf \eqref{eq: Joint_V1V2} to obtain the marginal. Panel (b) shows the absolute values of the errors between empirical and theoretical cdf of $V_2$. The upper plot of panel (c) shows the empirical univariate cdf of $V_1$ (in black), together with the asymptotic theoretical values (in gold), in log scale, obtained by numerically  integrating the asymptotic joint cdf \eqref{eq: Joint_V1V2}. The lower plot of panel (c) represents the $L_1$ error between the marginal empirical and theoretical cdf.}\label{fig: SC_comparison}
\end{figure}

According to Theorem \ref{th: Joint_max_of_conc}, since $1-\alpha=\beta$, we have $\tilde{a}_n=a_n=n$. Hence,
\[
\mathcal{L}(nx,ny)={\left(n^{\theta}x^{\theta}+n^{\theta}y^{\theta}-1\right)}^{-1/\theta}n{(xy)}^{1/2}\nto{\left(x^{\theta}+y^{\theta}\right)}^{-1/\theta}{(xy)}^{1/2}=\tilde{c}(x,y).
\]
Note that, for each $a>0$, 
\[
    \tilde{c}(ax,ay)={\left({(ax)}^{\theta}+{(ay)}^{\theta}\right)}^{-1/\theta}{(a^2xy)}^{1/2}=\tilde{c}(x,y),
\]
which means that $\tilde{c}(x,y)$ and also $r(x,y):=2^{1/\theta}\tilde{c}(x,y)$ satisfy \eqref{eq: BSV_definition}. Moreover it is straightforward to show that $\tilde{c}(x,y)$ satisfies Condition \eqref{cond: limit of L(n,a_n)} in Theorem \ref{th: Joint_max_of_conc}.
Since it holds
\[
    \frac{\partial}{\partial x}r(x,y)=%
    2^{1/\theta} \frac{\partial}{\partial x}\tilde{c}(x,y)=%
    2^{1/\theta}x^{-1} \tilde{c}(x, y)\left(\frac 12 - {\left(1 + {\left(\frac yx\right)}^{\theta}\right)}^{-1}\right),
\]
applying Lemma \ref{lemma: F2_F3_limit} and Theorem \ref{th: Joint_max_of_conc} gives that
\begin{subequations}
    \begin{align}
H_1(y\mid x)
&=1-{\left(1+{\left(\frac{y}{x}\right)}^{\theta}\right)}^{-1/\theta},\label{eq: SClayton_H1}\\
H_2(y\mid x)&=%
\left(1-{\left(1+{\left(\frac{y}{x}\right)}^{\theta}\right)}^{-1/\theta-1}\right)%
\e^{-\frac{1}{y}+{\left(1+{\left(\frac{y}{x}\right)}^{\theta}\right)}^{-1/\theta}}\label{eq: SClayton_H2}.
    \end{align}
\end{subequations}
The plot in Figure~\ref{subfig: SC_comparison_surface} shows the bivariate surface computed by applying Theorem \ref{th: Joint_max_of_conc} (the integration is performed numerically) and the values of the bivariate empirical cdf~computed in the simulated $V_1$ and $V_2$.
These points lie on the surface, thus displaying the very good quality of the asymptotic approximation of Theorem \ref{th: Joint_max_of_conc}. It can also be assessed from Figures~\ref{subfig: SC_comparison_boxplot} and \ref{subfig: SC_comparison_margins}, showing the boxplot of the absolute value of the errors between the black cloud and the surface in Figure~\ref{subfig: SC_comparison_surface}, the asymptotic approximation of the marginal distributions of $V_1$ and the errors in $V_2$, respectively. 
For what concerns the error (lower plot of Figure~\ref{subfig: SC_comparison_margins}), we observe that it is small for low values of $v_2$ (the maximum error is close to $1\%$) and converges very quickly to $0$. In particular, it is $0\%$ for all the values of $v_2$ we are considering.

\subsection{Unit Fréchet marginals with Logistic Extremal copula}\label{subsec: counterex_JN95_Logistic}

Let $(X,Y)$ be a random vector with unit Fréchet marginals and a bivariate symmetric logistic extremal dependence, \emph{i.e.}
\[
    F(x,y)=\e^{-V(x,y)},\]
where
\[
    V(x,y)={\left( x^{-1/\gamma} + y^{-1/\gamma}\right)}^\gamma,\,x,y>0,\,\gamma\in(0,1).
\]
By considering the bivariate joint tail approximation~\eqref{eq: Ledford98_joint-tail} proposed in \cite{Ledford1997} characterised by a general extremal dependence function $V$, we can write in our case
\begin{equation}
    \label{eq: L_Logistic}
    \calL(x,y) = {(xy)}^{-1/2}{(x+y)}-{(xy)}^{1/2}V(x,y).
\end{equation}
\begin{figure}[t]
    \centering
    \subfloat[][\emph{Bivariate distribution of $(V_1,V_2)$.}\label{subfig: Logistic_comparison_surface}]{\includegraphics[trim={5cm 1cm 4cm 2.5cm}, clip, width = 0.37\linewidth]{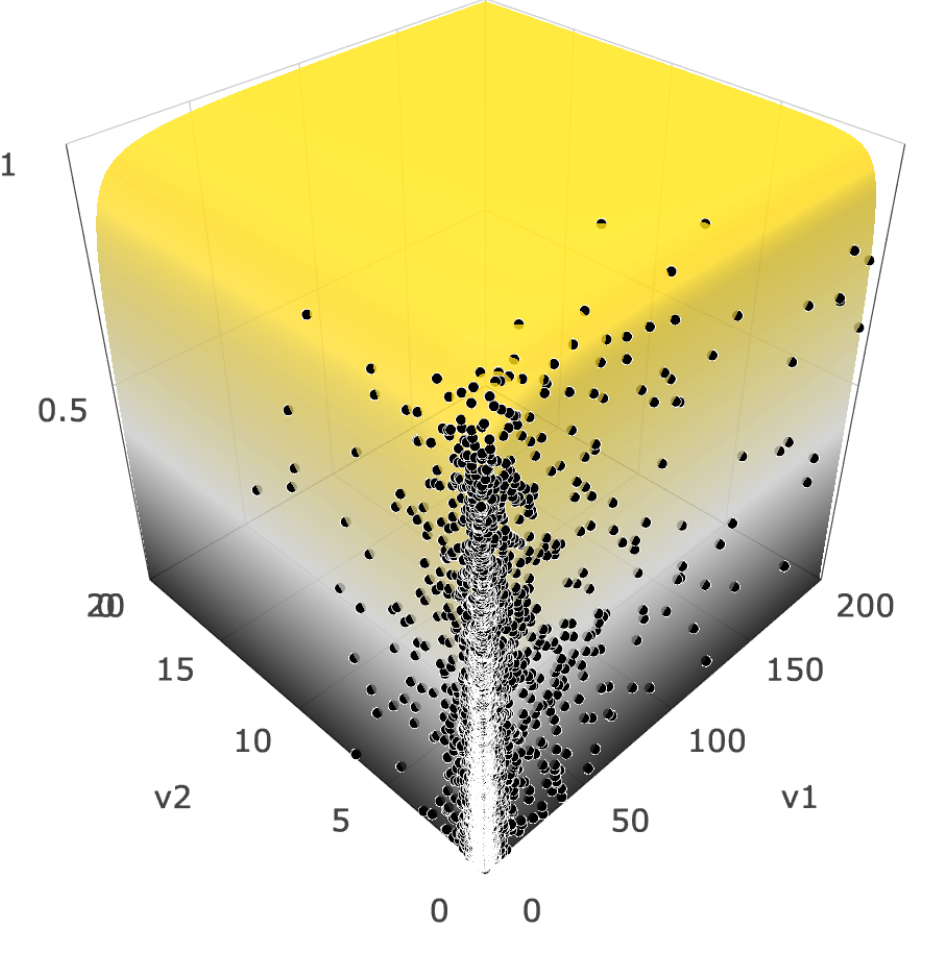}}\quad
    \subfloat[][\emph{Error evaluation.}\label{subfig: Logistic_comparison_boxplot}]{\includegraphics[trim={5cm 0cm 5cm 0cm}, clip, width = 0.168\linewidth]{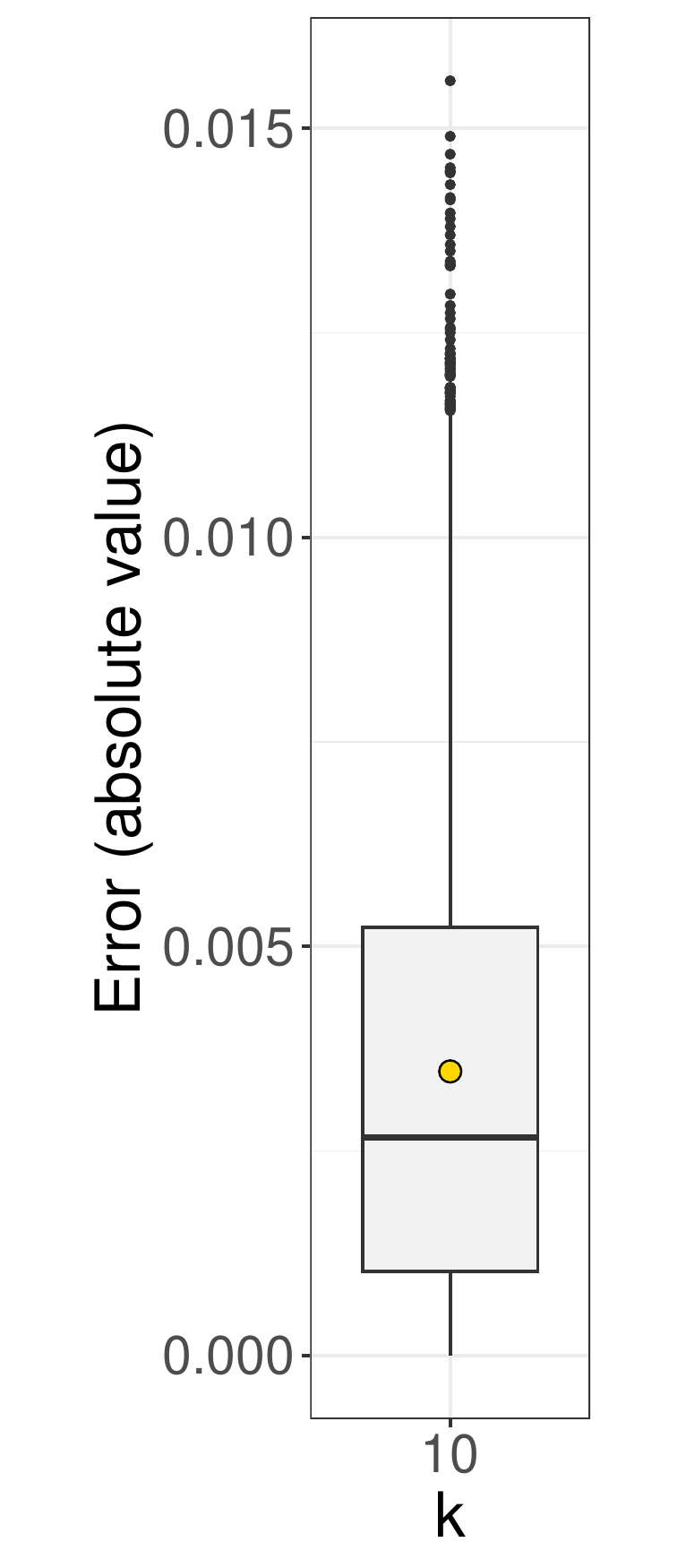}}
    \quad
    \subfloat[][\emph{Marginal distributions diagnostics.}\label{subfig: Logistic_comparison_marginals}]{\includegraphics[width = 0.39\linewidth]{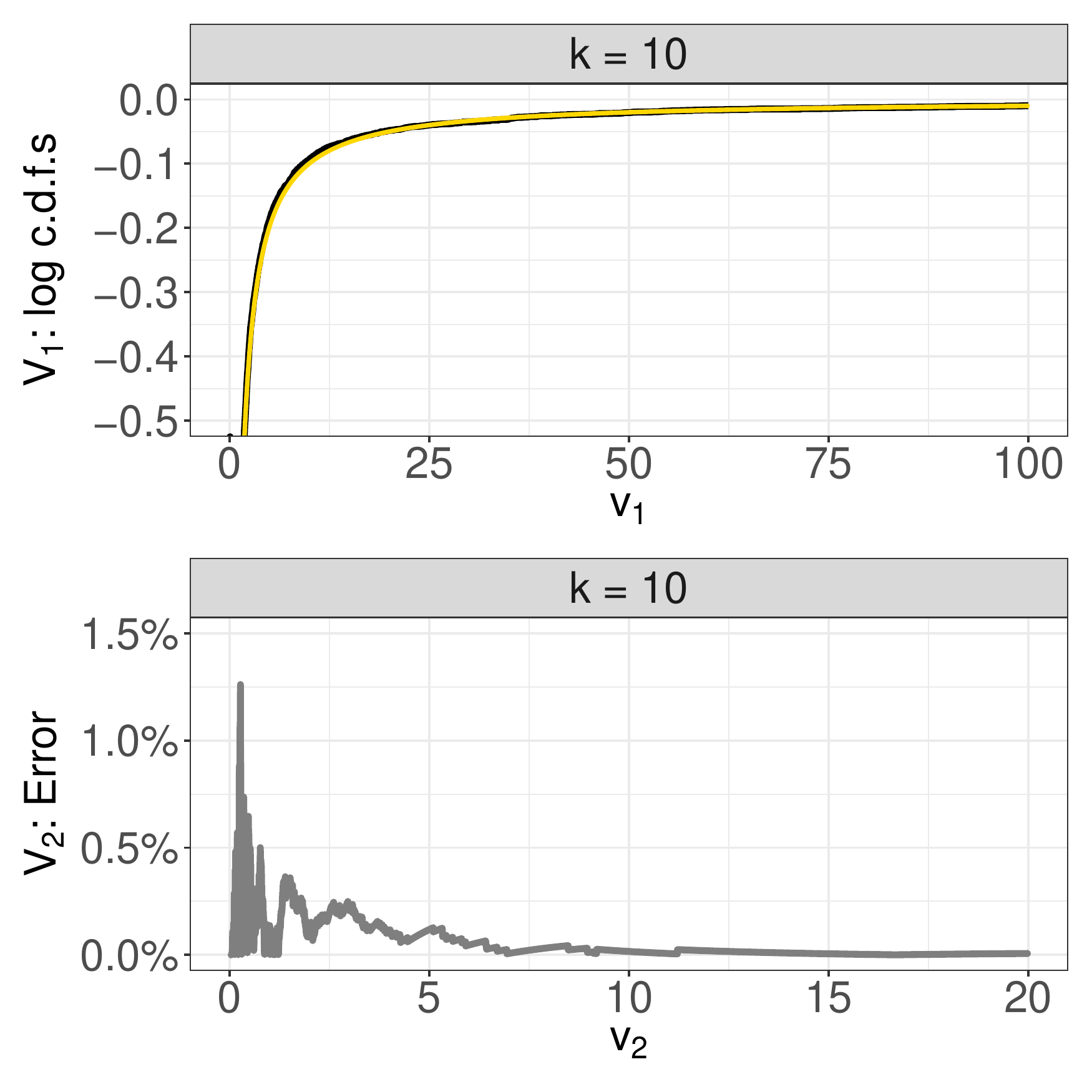}}
    \caption{Logistic$(\gamma = 1/2)$, with $k = 10$. Panel (a) shows the empirical univariate log cdf~of $V_1$ (in black), together with the asymptotic theoretical log values (in gold) obtained by numerically integrating the asymptotic joint cdf~\eqref{eq: Joint_V1V2} to obtain the marginal one. Panel (b) shows the absolute values of the errors between the empirical and the theoretical cdf of $V_2$. The upper plot of panel (c) shows the empirical univariate cdf of $V_1$ (in black), together with the asymptotic theoretical values (in gold), in log scale, obtained by numerically  integrating the asymptotic joint cdf~\eqref{eq: Joint_V1V2}. The lower plot of panel (c) represents the $L_1$ error between the marginal empirical and theoretical cdf.}\label{fig: Logistic__comparison}
\end{figure}
Note that
\[
\mathcal{L}(n,n)=2-V(1,1),
\]
since $V$ is homogeneous of order $-1$ (\emph{i.e.}~$V(nx,ny)=n^{-1}V(x,y)$).
Moreover, $1-\alpha=\beta$, so that $\tilde{a}_n=a_n=n$. Hence, we can write
\[
\begin{split}
\mathcal{L}(nx,ny)&=n^{-1}{(xy)}^{-1/2} n(x+y) - n{(xy)}^{1/2}n^{-1}V(x,y)\\
&={(xy)}^{-1/2}{(x+y)}-{(xy)}^{1/2}V(x,y)=\tilde{c}(x,y).
\end{split}
\]
Note that, for each $a>0$, 
\[
    \tilde{c}(ax,ay)=\tilde{c}(x,y),
\]
so that the candidate $\calL$ in \eqref{eq: L_Logistic} is a proper BSV function. Simple algebra helps to prove that $\tilde{c}(x,y)$ satisfies Condition \eqref{cond: limit of L(n,a_n)} in Theorem \ref{th: Joint_max_of_conc}.
Since it holds
\[
    \frac{\partial}{\partial x}\tilde{c}(x,y)=%
    \frac{{(xy)}^{-1/2}}{2}\left(1 - \frac yx - yV(x,y)\left(1 - 2{(xV(x,y))}^{-1/\gamma}\right)\right),
\]
by applying Lemma \ref{lemma: F2_F3_limit} and Theorem \ref{th: Joint_max_of_conc}, we obtain
\begin{subequations}
    \begin{align}
H_1(y\mid x)
&= x\left( V(x,y) -y^{-1}\right),\label{eq: Extremal_Logistic_H1}\\
H_2(y\mid x)&=%
{\left(1+{\left(\frac{y}{x}\right)}^{- \frac{1}{\gamma}}\right)}^{\gamma-1}%
\e^{-\frac 1x\left({\left(1+{\left(\frac{y}{x}\right)}^{-\frac{1}{\gamma}}\right)}^\gamma-1\right)}\label{eq: Extremal_Logistic_H2}.
    \end{align}
\end{subequations}
After having obtained the functions $H_1$ and $H_2$, we can now apply a numerical integration to compute the joint cdf~\eqref{eq: Joint_V1V2}.
%
%%%%%%%%%%%%%%%%%%%%%%%%%%%%%%%%%%%%%%%%%%%%%%
%% Multiple Appendixes:                     %%
%%%%%%%%%%%%%%%%%%%%%%%%%%%%%%%%%%%%%%%%%%%%%%
%\begin{appendix}
\section{Discussion on the Gaussian case}\label{sec: Gaussian case}
Whilst our focus so far was on obtaining the joint asymptotic cdf of $(V_1,V_2)$ when $(X,Y)$ exhibits asymptotic or extremal dependence, we question if Condition~\eqref{cond: limit of L(n,a_n)} of Theorem~\ref{th: Joint_max_of_conc} is strongly related to the extremal dependence assumption.
In other words, what does occur if the BSV function $\calL$ goes to $0$ or to infinity? This is why we return to the case when $(X,Y)$ is Gaussian, as it is well known that $(X,Y)$ is then asymptotically independent.
This case has already been treated in the literature, due both to practical interest of the distribution and to its known manageability. Although direct computations in this case become simple, problems arise in applying models involving approximations with SV functions. In fact, even if such approximations are provided, they are then not used to compute for instance the conditional distributions we encountered above (as far as we could observe in the literature). The objective of the present section is to question this apparent paradox and provide tools to overcome this issue, thus bridging the gap between direct computations and those made through asymptotic models, with some attention put on the asymptotic behaviour of the BSV $\calL$.
%
%%%%%%%%%%%%%%%%%%%%%%%%%%%%%%%%%%%%%
\subsection{Existing literature}
%%%%%%%%%%%%%%%%%%%%%%%%%%%%%%%%%%%%%
%
Let us start to review the existing literature on the Gaussian case.
\begin{enumerate}[leftmargin=*, label=(\alph*)]
    \item \cite{Joshi1995}: By direct computations, the authors are able to provide constants $\tlan,\,\tlbn$ such that the rescaled distribution of $(V_1,V_2)$ converges to a product of a standard normal cdf to the power $k$ and a Gumbel cdf. When applying those constants in our case, after the transformation of the Gaussian marginals into unit Fréchet, Condition~\eqref{cond: limit of L(n,a_n)} in Theorem~\ref{th: Joint_max_of_conc} is not satisfied. In fact, we obtain that $F_1$ degenerates for $n\to\infty$.
    
    Hence, our question: Could we find other transformations of the marginals that let us circumvent the case  $\calL\to 0$, and possibly find a new method or solution to handle this case?
    The relevance of this problem lies in the fact that relations between asymptotic dependence and the behavior of $\calL$ are generally unknown:
    We know by formula \eqref{eq: relation_lambda_U_eta} that the behaviour of $\calL(n,n)$ and the values $\eta$ takes are linked to the upper-tail dependence coefficient $\lambda_u$, but the case $\calL\not\to 0$ and $\eta<1$ is not tackled. Moreover, we also need to consider $\calL(n,\tilde a_n)$ instead of $\calL(n,n)$.
    \item \cite{Ledford1998}: The focus here is only on the asymptotic distribution of the concomitant of the maximum $X_{(n)}$, that is the asymptotic distribution of $V_1$ when $k=1$ (in our notation), meaning to study $F_1$. The authors explain that there can be cases where unit Fréchet marginals are not appropriate and suggest a standard Gumbel transformation for the variable $Y$, arriving to a new BSV function $\calL_1$ (in fact the first order approximation of $\calL$) that avoids the degeneracy of the limit of $H_1$ (with our notations). However, to compare this approach with the case where both the marginals are unit Fréchet, they need to assume $y\to\infty$.
    \end{enumerate}

Let $(X_N,Y_N)$ be a bivariate normal random vector with standard normal marginals and with correlation $0<\rho<1$. Theorem 3 in \cite{Joshi1995} states that, in the extreme case ($k$ fixed, $n\to\infty$), it holds that
\begin{equation}\label{eq: JN95_Theorem_3_joint_V1V2_Gaussian_extreme}
\Pro{\frac{V_1-\tilde{a}_n}{\tilde{b}_n}\leq v_1,\frac{V_2-a_n}{b_n}\leq v_2}\underset{n\to\infty}{\to} \Phi^k(v_1)\e^{-\e^{-v_2}},
\end{equation}
where 
\begin{subequations}
    \begin{align}
        a_n &= {\left(2\log n\right)}^{-1/2} & b_n &= \sqrt{2\log n} - \frac 12 \frac{\lln}{\sqrt{2\log n}} \label{subeq: Gaussian_a_b}\\
        \tlan &= \sqrt{\omr} & \tlbn &= \rho b_n. \label{subeq: JN95_tilde_a_b}
    \end{align}
\end{subequations}
Formula \eqref{eq: JN95_Theorem_3_joint_V1V2_Gaussian_extreme} represents an asymptotic independence property of $V_1$ and $V_2$, strongly connected to the fact that
\begin{equation}\label{eq: JN95_cdf_F1}
    F_1(\tlan y + \tlbn \mid a_n x + b_n) = \Pro{Y_N \leq \tlan y + \tlbn \mid X_N > a_n x + b_n} \nto \Phi(y).
\end{equation}
Recall that our model in Theorem~\ref{th: Joint_max_of_conc} is not applicable here. In fact, Condition \eqref{cond: limit of L(n,a_n)} is not satisfied as $\calL(n, \tilde{a}_n)\to 0$ for any choice of the normalizing constants.
Thus, further study would be needed in such a case.

A way to look at the problem may be through \cite{Wadsworth2013}. The authors study the effect of letting the components of a bivariate random vector grow at different rates, with the result of stating a new class of regular variation conditions and providing a new characterization of the link between the multivariate tail decay and the considered marginal growth rates.

More precisely, the authors transform the marginals of the original vector into standard exponentials, called $(X_E,Y_E)$, and assume that, for all $(\zeta_1,\zeta_2)\in\R_+^2\setminus\{\bzero\}$,
\begin{equation}\label{cond: WT2013_tail_dep_assumption}
    \Pro{X_E>\zeta_1\log n,Y_E>\zeta_2\log n} \simn L(n;\zeta_1,\zeta_2)n^{-\kappa(\zeta_1,\zeta_2)}
\end{equation}
where $L$ is a univariate SV function in $n$ as $n\to\infty$, and the function $\kappa(\zeta_1,\zeta_2)$ \emph{maps the different marginal growth rates to the joint tail decay rate} with the role of providing information \emph{about the level of dependence between variables at sub-asymptotic levels}.
Assumption \eqref{cond:  WT2013_tail_dep_assumption} is then used to describe models for conditional probabilities of extreme exceedances, \emph{i.e.} 
\begin{equation}\label{cond: WT2013_cond_tail_dep}
    \Pro{X_E>\zeta_1\log n\mid Y_E>\log n} \simn L(n;\zeta_1,1)\e^{(1-\kappa(\zeta_1,\zeta_2))\log n}.
\end{equation}
Studying \eqref{cond: WT2013_cond_tail_dep} is proven by the authors to be equivalent to the search for functions $a^*,\,b^*$ and $h$ such that
\begin{equation}\label{cond: WT2013_abk_cond_tail_dep}
    \lim_{n\to\infty}\Pro{\frac{X_E - b^*(\log n)}{a^*(\log n)}>x\con Y_E>\log n} = \e^{-h(x)}.
\end{equation}
The use of the exponential marginals is helpful to identify the normalizing constants for the bivariate Gaussian example, thus partially filling the gap left open in \cite{Ledford1998}, where, as previously explained, a way to tackle the problem is given in the case of Fréchet-Gumbel marginals, albeit that solution can be used only for some high values of $Y$.
However, this discussion does not yet answer our main  question: Why, when using unit Fréchet marginals, aren't we able to find normalizing constants that lead to a non-degenerate limit for $\calL(nx,\tilde{a}_ny+\tilde{b}_n)$?
Our understanding of the problem is made more complete by the insightful paper by \cite{Heffernan2007}. There, the authors explore the implications of assuming the existence of a
scaling function $c^*_1(\cdot)$, a centering function $c^*_2(\cdot)$, and a non-null Radon measure $\mu$ on Borel subsets of
$[-\infty,\infty]\times (-\infty, \infty]$ such that, for each fixed $y\in E_\xi:=\{y\in\R: 1+\xi y>0\}$ (where $\xi$ is the shape parameter of the extreme value distribution relative to $F_Y$), it holds
\begin{enumerate}[label=(\alph*)]
    \item $\mu([-\infty,x]\times (y,+\infty))$ is a finite and non-degenerate distribution function in $x$,
    \item the limit of the cdf~of the normalized $X$ given that $Y$ is extreme exists and
    \begin{equation}\label{eq: HR07_cond_limit}
        \Pro{\frac{X - c^*_2\circ b^{\leftarrow}(t)}{c^*_1\circ b^{\leftarrow}(t)}\leq x\con Y>t}\underset{t\to y^+}{\to} \mu([-\infty,x]\times (0,+\infty)),
    \end{equation} 
\end{enumerate}
where $y^+$ is the right-end point of $F_Y$.
It is worth noticing the fact that we are free to change the marginal of $Y$ without disturbing the convergence in \eqref{eq: HR07_cond_limit}, but this is not true for the $X$ variable.
In the paper it is proven that, in the case of the bivariate standard Gaussian, it is not possible to obtain a non-degenerate expression for $\mu$ in \eqref{eq: HR07_cond_limit} if the marginals of $X$ and $Y$ are both transformed into the Pareto type, while an asymptotic form is still reachable if the marginals are transformed into exponentials. We refer to \cite[Sec.~7]{Heffernan2007} for a thorough discussion of this aspect.
%
%%%%%%%%%%%%%%%%%%%%%%%%%%%%%%%%%%%%%
\subsection{Revisiting the Gaussian case}
%%%%%%%%%%%%%%%%%%%%%%%%%%%%%%%%%%%%%
Now that we have a better and more complete understanding on the issues underlying the treatment made in the literature of the Gaussian case, we provide a way to find suitable constants through the application of model \eqref{cond: WT2013_cond_tail_dep} and formulas \eqref{subeq: Gaussian_a_b} and \eqref{subeq: JN95_tilde_a_b}. We know that, if $(X_N,Y_N)$ is a bivariate normal random vector and correlation $\rho\in(0,1)$,
\begin{equation}\label{eq: JN95_tail_F1}
    \Pro{Y_N > \tlan y + \tlbn \mid X_N > a_n x + b_n} \nto 1-\Phi(y),
\end{equation}
where the constants are given in \eqref{subeq: Gaussian_a_b} and \eqref{subeq: JN95_tilde_a_b}. The following results hold:
\begin{theorem}\label{th: tilde_constants_Gaussian}
    Let $(X_E, Y_E)$ be a bivariate random vector with unit exponential marginals and correlation $\rho\in (0,1)$, such that \eqref{cond: WT2013_cond_tail_dep} holds, with $\zeta_1:=\zeta_1(n,y)$. Then, we have
    \begin{equation}
        \Pro{Y_E > \tilde{a}_{n,E}y + \tilde{b}_{n,E} \mid X_E > \log n} \to \frac{1}{y\sqrt{2\pi}}\e^{-{(\rho y)}^2/2},\quad\text{as } n\to\infty,
    \end{equation}
    where
    \begin{subequations}
        \begin{align}\label{eq: exp tilde a b with JN}
        \tilde{a}_{n,E} &=  \rho\sqrt{\omr}\left(\sqrt{2\log n} - \frac 12 \frac{\lln}{\sqrt{2\log n}}\right)\\
        \text{and\hspace*{0.5cm}}\tilde{b}_{n,E} &= \rho^2\log n - \frac{\rho^2}{2}\lln + \frac{\rho^2}{16}\frac{{(\lln)}^2}{\log n}.
    \end{align}
    \end{subequations}
    \end{theorem}
    \begin{corollary}\label{coro: tilde_constants_Gaussian}
    Note that, if we assume $y$ large enough (as in \cite{Ledford1998}), then we can use the Mill's ratio to obtain
    \begin{equation}\label{eq: tail_Gaussian_approx}
        \Pro{Y_E > \tilde{a}_{n,E}y + \tilde{b}_{n,E} \mid X_E > \log n} \simn \rho\left(1 - \Phi(\rho y)\right).
    \end{equation}
    \end{corollary}

Refer to Appendix~\ref{app: proofs} for the proofs.

The limit distribution given in \eqref{eq: tail_Gaussian_approx} is similar to those in \cite{Joshi1995} and \cite{Wadsworth2013}. The differences among the final formulas are due to the different chosen orders of the Taylor approximations.
%----------------------------------------------------------------------------------------
\section{Discussion}
The work presented has provided a deep understanding of the impact of the extremal dependence of a bivariate random vector $(X,Y)$ on the joint asymptotic behavior of suitably rescaled maxima over subsets of concomitants. Studying the framework of asymptotically dependent random variables is actually of great interest in practical applications and completes the literature.

Our main result addresses this problem, providing the expression of the asymptotic joint cdf of the two rescaled maxima of concomitants.
In particular, we take advantage of the joint-tail model proposed in \cite{Ledford1997} to provide the asymptotic form of all the conditional distributions describing $Y$ given the behaviour of $X$, which contribute to the asymptotic joint cdf of the two maxima of concomitants. We also explain in detail how to choose suitable normalizing constants needed to avoid the degeneracy of such distributions (thus preventing the joint cdf of the maxima of concomitants to be degenerate as well). Our theoretical results are illustrated through two examples of interest in the risk analysis literature: We compute the analytical expressions of the conditional distribution and perform a numerical integration to obtain the joint cdf. We show the goodness of the asymptotic approximation in all the examined cases through a numerical simulation.

Furthermore, we reconsider Theorem 2 in \cite{Joshi1995} and show how some hypotheses are connected to the upper-tail dependence coefficient characterizing the copula of the parent vector $(X,Y)$, thus providing an intuitive interpretation of the formulas therein and a new reformulation of the cited Theorem.

More importantly, we revisit the example of the bivariate Gaussian distribution in our case, furnishing a comprehensive discussion on it. We deeply study the reasons why the joint-tail model in \cite{Ledford1997} fails in providing a non-degenerate asymptotic distribution and fill the gap in the extension of the joint-tail model proposed in \cite{Wadsworth2013}, by showing how to compute suitable normalization constants in this case.
Our results are coherent with those obtained by direct computation in \cite{Joshi1995}.
%
%%%%%%%%%%%%%%%%%%%%%%%%%%%%%%%%%%%%%%%%%%%%%%
%% Example with multiple Appendixes:        %%
%%%%%%%%%%%%%%%%%%%%%%%%%%%%%%%%%%%%%%%%%%%%%%
\begin{appendix}
    \section{Appendix - Proofs}\label{app: proofs}%\subsection{Proofs}
%----------------------------------
\begin{proof}[Proof of Lemma \ref{lemma: nP to 0 equiv F_3 to 1}]
    Using the notations of Lemma \ref{lemma: nP to 0 equiv F_3 to 1}, and indicating with $\mathscr{D}(G_{\gamma_1})$, $\mathscr{D}(G_{\gamma_2})$, $\mathscr{D}(G_0)$ the MDAs of the Fréchet, the Weibull and the Gumbel distributions, respectively,
    we can write, for $x,y\in\R$ and $x^+$ denoting the upper end-point of the distribution of $X$,
    \[
    \begin{split}
    &n\Pro{X>a_nx+b_n,Y>c_ny+d_n}\\
    &=\int_{a_nx+b_n}^{x^+}n\Pro{Y>c_ny+d_n\mid X=t}f_X(t)\dd{t}\\%\\
    &\leq n\int_{a_nx+b_n}^{x^+} f_X(t)\dd{t}=n(1-F_X(a_nx+b_n))<\infty,\,\forall n\in\N.
    \end{split}
    \]
    As shown in the proof of Proposition 2.5 in \cite{Resnick1987}, it holds
    \[
    na_nf_X(a_nz+b_n)\indic\left(x\leq z<a_n^{-1}(x^+-b_n)\right)
    \nto
    \indic\left(z\geq x\right)g(z)
    \]
    with
    \[
    g(z):=\begin{cases}
    \gamma_1 z^{-(\gamma_1+1)},\, z>0,\,\gamma_1>0, & \text{if } F_X\in \mathscr{D}(G_{\gamma_1}),\\
    \gamma_2 (-z)^{\gamma_2-1},\, z<0,\, \gamma_2>0, & \text{if } F_X\in\mathscr{D}(G_{\gamma_2}),\\
    e^{-z},\,z\in\R, & \text{if } F_X\in\mathscr{D} (G_0).
    \end{cases}
    \]
    
    Note the role of the indicator $\indic\left(x\leq z<a_n^{-1}(x^+-b_n)\right)$, needed to ensure that we are integrating over the support of $X$.
    In order to apply the Extended Dominated Convergence Theorem, we need to check that 
    \[
    \int_\R na_nf_X(a_nz+x^+)\indic\left(z\geq x\right)\dd{z}=n\left(1-F_X(a_nx+b_n)\right)\nto \int_\R g(z)\indic\left(z\geq x\right)\dd{z}.
    \]
    
    For the Gumbel case, we have
    \[
    n\left(1-F_X(a_nx+b_n)\right)\to \e^{-x} \text{ and }\int_\R g(z)\indic\left(z\geq x\right)\dd{z}=\e^{-x}.
    \]
    The Fréchet and Weibull cases follow the same reasoning. Since the two quantities are equal, we can exchange limit and integral, as follows:
    \[
    \begin{split}
    &\lim_{n\to\infty}\int_\R n\Pro{Y>c_ny+d_n\mid X=t}f_X(t)\indic\left(t\geq a_nx+b_n\right)\dd{t}\\
    &=\int_\R\lim_{n\to\infty}\Pro{Y>c_ny+d_n\mid X=a_nz+b_n}\cdot\lim_{n\to\infty}na_nf_X(a_nz+b_n)\\
    &\qquad\qquad\cdot\indic\left(z<a_n^{-1}(x^+-b_n)\right)\cdot\indic\left(z\geq x\right)\dd{z}.
    \end{split}
    \]
    Since the integrand is positive, then the integral tends to $0$ if and only if the latter tends to $0$, that is,
    if and only if 
    \[
    \Pro{Y>c_ny+d_n\mid X=a_nz+b_n}=1-F_3\left(c_ny+d_n\mid a_nz+b_n\right)\nto 0
    \]
    for all $z\geq x$ (note that $\displaystyle{\lim_{n\to\infty}na_nf_X(a_nz+b_n)\neq 0}$). Hence the result.
    \end{proof}
\begin{proof}[Proof of Lemma \ref{lemma: F2_F3_limit}]
    In order to make use of \eqref{eq: Ledford98_joint-tail}, we write all the conditionals in terms of $\bar{F}$ and recall the notation $C_{\not\ind}$ for the case $\alpha + \beta=1\text{ and }\lambda_U>0$.
    %Since we consider unit Fréchet margins, we know that $a_n=c_n=n,\,b_n=d_n=0$.
    For what concerns $F_2(y\mid x)$, we have
    \[
        F_2(y\mid x) =\Pro{Y\leq y\mid X<x} = 1 - \Pro{Y> y\mid X<x} = 1 - \frac{\bar{F}_Y(y) - \bar{F}(x,y)}{F_X(x)}.
    \]
    Since
    \[
        \begin{split}
            &\frac{\bar{F}_Y(ny) - \bar{F}(nx,ny)}{F_X(nx)}= \e^{\frac{1}{nx}}\left( 1 - \e^{-\frac{1}{ny}}-\frac{\calL(nx,ny)}{n^{\alpha+\beta}x^{\alpha}y^{\beta}}\right)\\
            &\simn \frac 1n \left( 1 + \frac 1{nx}\right)\left(\frac 1y - \frac{\calL(nx,ny)}{n^{\alpha+\beta-1}x^{\alpha}y^{\beta}}\right)= \frac 1{ny} \left( 1 + \frac 1{nx}\right)\left(1 - \frac{\calL(nx,ny)}{n^{\alpha+\beta-1}x^{\alpha}y^{\beta-1}}\right)\\
            &\simn \frac 1{y} \left( 1 + \frac 1{nx}\right)\left(1 - \frac{\calL(n,n)}{n^{\alpha+\beta-1}}r(x,ny)x^{-\alpha}y^{1-\beta}\right)\simn \frac 1{ny} \left(1 - \lambda_U r(x,y){\left(\frac yx\right)}^{\alpha}\indic_{C_{\not\ind}}\right),
        \end{split}
    \]
    we obtain
    \[
        F^{n}_{2}(ny\mid nx)= {\left[1 - \left(\frac{\bar{F}_Y(ny)}{F_X(nx)} - \frac{\bar{F}(nx,ny)}{F_X(nx)}\right)\right]}^n
        \simn {\left[1-\frac 1{ny} \left(1 - \lambda_U r(x,y){\left(\frac yx\right)}^{\alpha}\indic_{C_{\not\ind}}\right)\right]}^n.
    \]
Thus 
\[
    F^{n}_{2}(ny\mid nx)\nto \e^{-\frac 1{y} \left(1 - \lambda_U r(x,y){\left(\frac yx\right)}^{\alpha}\indic_{C_{\not\ind}}\right)}.
\]
For $F_3(y\mid x)$, note that
\[
    \begin{split}
        F_3(y\mid x) &= \frac{1}{f_X(x)}\frac{\dpar}{\dpar x}F(x,y)=\frac{1}{f_X(x)}\frac{\dpar}{\dpar x}\left(F_X(x)-\bar{F}(y)+\bar{F}(x,y)\right)\notag\\
        &=1+\frac{1}{f_X(x)}\frac{\dpar}{\dpar x}\bar{F}(x,y)
        = 1+ x^{2 - \alpha}y^{-\beta}\e^{\frac 1x}\left( \frac{\dpar}{\dpar x}\calL(x,y) - \frac{\alpha}{x}\calL(x,y) \right).
    \end{split}
\]
It follows that
\[
\begin{split}
F_3(ny\mid nx)&\simn 1+n^{2-\alpha-\beta}\e^{\frac{1}{nx}}x^{2-\alpha}y^{-\beta}\left( \frac{\dpar}{\dpar (nx)}\calL(nx,ny) - \frac{\alpha}{nx}\calL(nx,ny)\right)\\
&= 1+ \e^{\frac{1}{nx}}x^{2-\alpha}y^{-\beta}n^{2-\alpha-\beta}\left( \frac 1n\frac{\dpar}{\dpar x}\calL(nx,ny) - \frac{\alpha}{nx}\calL(nx,ny)\right)\\
&\simn 1+ \e^{\frac{1}{nx}}n^{1-\alpha-\beta}\calL(n,n)x^{2-\alpha}y^{-\beta}\left( \frac{\dpar}{\dpar x}r(x,y) - \frac{\alpha}{x}r(x,y)\right)\\
&\nto 1+\lambda_U x^{2-\alpha}y^{-\beta}\left( \frac{\dpar}{\dpar x}r(x,y) - \frac{\alpha}{x}r(x,y)\right)\indic_{C_{\not\ind}},
\end{split}
\]
which concludes the proof.
\end{proof}
\begin{proof}[Proof of Theorem \ref{th: Joint_max_of_conc}]
    The proof is based on Lemma~\ref{lemma: F2_F3_limit}, and \cite{Joshi1995}. To begin with, we report the finite-sample bivariate cumulative distribution of $(V_1,V_2)$ (\cite[Equation (2.1)]{Joshi1995}),
    % we need to compute
    \begin{equation}\label{eq: V1_V2_finite_sample_cdf}
        F_{(V_1,V_2)}(v_1,v_2)=\E{h\left(v_1,v_2,X_{(n-k)}\right)}
    \end{equation}
    where the function $h$ is defined as
    \[
        h(v_1,v_2,x)=F_1^k(v_1\mid x)F^{n-k-1}_2\left(v_2\mid x\right)F_3\left(v_2\mid x\right).
    \]
    The rational under Equation \eqref{eq: V1_V2_finite_sample_cdf} lies in \cite{Kaufman1992}. Conditioned on the event $\{X_{(n-k)}=x\}$, the maxima $V_1$ and $V_2$ have the following behaviour: $V_1$ can be considered as the sample maximum of a random sample of size $k$ from $F_1(\cdot\mid x)$, while $V_2$ behaves like the sample maximum of a sample of size $n-k$, composed by $n-k-1$ i.i.d.~variables from $F_2(\cdot\mid x)$ and a single random variable from $F_3(\cdot\mid x)$.

    Following \cite{Joshi1995}, from Equation \eqref{eq: V1_V2_finite_sample_cdf}, we can write
    \[
    \begin{split}
        F_{(V_1,V_2)}(\tilde{a}_n v_1+\tilde{b}_n,\,c_n v_2+d_n)=\int &F_1(\tilde{a}_n v_1+\tilde{b}_n \mid a_n x+b_n)F^{n-k-1}_2\left(c_n v_2+d_n\mid a_nx+b_n\right)\\
        &F_3\left(c_n v_2+d_n\mid a_nx+b_n\right)a_nf_X(a_nx+b_n)\dd{x}.
    \end{split}
    \]
    Since our goal is to compute the joint asymptotic distribution of $V_1$ and $V_2$ as $n\to\infty$, we refer to \cite{Joshi1995} and \cite{Nagaraja1994} for what concerns the swap between the limit operation and the integral.
    Since $X$ is unit Fréchet distributed, it is clear that the marginal $F_X$ is in the domain of attraction of $G_X(x)=\e^{-1/x}$ and that $g_X(x)=x^{-2}\e^{-1/x}$. We refer to Lemma 1 and Result 1 in \cite{Nagaraja1994}, for the convergence as $n\to\infty$ of the term $a_nf_X(a_nx+b_n)$ to $g_{(k)}(x):= \frac{{(-\log G_X(x))}^k}{k!}g_X(x)$.
    Theorem \ref{th: Joint_max_of_conc} then follows by Condition \eqref{cond: limit of L(n,a_n)} and Lemma \ref{lemma: F2_F3_limit}. 
\end{proof}
\begin{proof}[Proof of Theorem \ref{th: tilde_constants_Gaussian}]
Note that $X_N > a_n x + b_n$ is equivalent to $X_N > b_n$, as $a_n\to 0$ when $n\to\infty$.
Let us now report the model in \cite{Wadsworth2013} in terms of standard normal marginals, namely
\begin{equation}\label{eq: WT_Gaussian_model}
    \begin{split}
      &\Pro{Y_E > \zeta_1(n, y) \log n \mid X_E > \log n} \\&= \Pro{Y_N > \Phi^{-1}\left( 1 - n^{-\zeta_1(n, y)}\right) \con X_N > \Phi^{-1}\left( 1 - n^{-1}\right)}\\
      &\sim L_1\left( n; \zeta_1(n, y), 1\right) n^{1 - \kappa(\zeta_1(n, y), 1)}  
    \end{split}
\end{equation}
where
\begin{align}
    & \kappa(\zeta_1(n, y),1) = \frac{\zeta_1(n, y) + 1 -2\rho \sqrt{\zeta_1(n, y)}}{\omr},\qquad\text{and}\label{eq: WT_Gaussian_kappa}\\
    & L_1(n; \zeta_1(n, y), 1) = {(4\pi\log n)}^{\frac{2\rho^2 - \rho\left( \sqrt{\zeta_1(n,y)} + 1/\sqrt{\zeta_1(n,y)}\right)}{2(\omr)}}
        \frac{{\zeta_1(n,y)}^{\frac{1 - \rho/\sqrt{\zeta_1(n,y)}}{2(\omr)}}{(\omr)}^{3/2}}{(\sqrt{\zeta_1(n,y)}-\rho)(1 - \rho\sqrt{\zeta_1(n,y)})}\label{eq: WT_Gaussian_L1},
\end{align}
whenever $\rho^2<\min\{\zeta_1(n, y), 1\}$. The authors explain that the SV function $L$, which we have in the case of the standard exponential marginals (from Equation~\eqref{cond: WT2013_cond_tail_dep}), is such that $L/L_1\sim 1$ as $n\to\infty$.

To ease the notation, in the following, we use $\zeta_1$ for $\zeta_1(n,y)$.
We want to find some constants $\tilde{a}_{n,E}$ and $\tilde{b}_{n,E}$ such that $\{Y_E > \zeta_1(n, y) \log n\}$ is equivalent to $\{Y_E > \tilde{a}_{n,E}y + \tilde{b}_{n,E}\}$, so we set
\begin{equation}\label{eq: equiv beta - exp a,b tilde}
    \zeta_1 = \frac{\tilde{a}_{n,E}y + \tilde{b}_{n,E}}{\log n}.
\end{equation}
By using the first equality in \eqref{eq:  WT_Gaussian_model} and knowing that if $Y_N\sim\mathcal{N}(0,1)$ then $-\log\left( 1 - \Phi(Y_N)\right) \sim \text{Exp}(1)$, it follows that
\begin{equation}\label{eq: equiv beta - normal a,b tilde}
    \zeta_1 = \frac{-\log\left( 1 - \Phi(\tilde{a}_{n}y + \tilde{b}_{n}\right)}{\log n}.
\end{equation}
This means that we can rewrite \eqref{eq: JN95_tail_F1} by transforming the marginals into standard exponential and then apply the model \eqref{eq: WT_Gaussian_model}. Combining \eqref{eq: equiv beta - exp a,b tilde} and \eqref{eq: equiv beta - normal a,b tilde}, we obtain
\[
    \begin{split}
        \tilde{a}_{n,E}y + \tilde{b}_{n,E} &= -\log\left( 1 - \Phi(\tilde{a}_{n}y + \tilde{b}_{n}\right)\\
        &\simn -\log\left( \frac{\phi(\tilde{a}_{n}y + \tilde{b}_{n})}{\tilde{a}_{n}y + \tilde{b}_{n}}\right) \text{ through the Mill's ratio, if $\tlan\to\infty$ or $\tlbn\to\infty$}\\
        &= -\log\left( \frac{1}{\sqrt{2\pi}} \e^{-{(\tlan y + \tlbn)}^2/2}\right) + \log\left(\tilde{a}_{n}y + \tilde{b}_{n}\right)\\
        &= \frac{1}{2}\log(2\pi) + \frac{1}{2} \tlan^2 y^2 + \tlan\tlbn y + \frac{1}{2}\tlbn^2 + \log\tlbn + \log\left(1 + \frac{\tilde{a}_{n}}{\tilde{b}_{n}} y\right)\\
        &\simn \frac{1}{2}\log(2\pi) + \frac{1}{2} \tlan^2 y^2 + \tlan\tlbn y + \frac{1}{2}\tlbn^2 + \log\tlbn + \frac{\tilde{a}_{n}}{\tilde{b}_{n}}y \quad\text{if $\tlan = \smallo{\tlbn}$}\\
        &=\tlan\left(\tlbn + \frac 1\tlbn\right)y + \frac{1}{2}\log(2\pi) + \frac{1}{2} \tlan^2 y^2 + \frac{1}{2}\tlbn^2 + \log\tlbn\\
        &\simn \tlan\tlbn y + \frac{1}{2}\tlbn^2 \quad\text{(the discarded terms are $\smallo{\tlan\tlbn}$)}.
    \end{split}
\]
So, we can identify the terms, which gives
\[
    \tilde{a}_{n,E} =  \tlan\tlbn, \qquad \tilde{b}_{n,E} = \frac{1}{2}\tlbn^2.
\]
Note that, up to now, we have not used the constants in \eqref{subeq: JN95_tilde_a_b}, but found some general conditions on the their growth as $n$ increases.
By using \eqref{subeq: JN95_tilde_a_b}, we have
\begin{subequations}
    \begin{align*}%\label{eq: exp tilde a b with JN}
    \tilde{a}_{n,E} &=  \rho\sqrt{\omr}\left(\sqrt{2\log n} - \frac 12 \frac{\lln}{\sqrt{2\log n}}\right)\\
    \tilde{b}_{n,E} &= \rho^2\log n - \frac{\rho^2}{2}\lln + \frac{\rho^2}{16}\frac{{(\lln)}^2}{\log n}.
\end{align*}
\end{subequations}
Replacing the constants in \eqref{eq: equiv beta - exp a,b tilde} it yields
\begin{equation}\label{eq: expression of beta}
    \zeta_1 = \rho^2\left( 1 + \frac 1\rho \sqln y - \frac 12 \llln + \smallo{{(\log n)}^{-1}}\right).
\end{equation}
Now, notice that with $\zeta_1$ defined in \eqref{eq: expression of beta}, the condition $\rho^2<\min\{\zeta_1(n, y), 1\}$ is always satisfied for large $n$ and for $\rho\in(0,1)$, as it yields $\rho^2<\rho^2 + \bigO{{(\log n)}^{-1/2}}$.
Hence, since
%
%\begin{equation}\label{eq: expression of sqrt_beta}
\[
    \begin{split}
    \sqrt{\zeta_1} &= \rho{\left( 1 + \frac 1\rho \sqln y - \frac 12 \llln + \smallo{{(\log n)}^{-1}}\right)}^{1/2}\\
    &\simn \rho\left( 1 + \frac 1{2\rho} \sqln y - \frac 14 \llln  - \frac{\omr}{4\log n}y^2+ \smallo{{(\log n)}^{-1}}\right),
    \end{split}
\]
%\end{equation}
the expression of $\kappa(\zeta_1,1)$ becomes
\[
    \kappa(\zeta_1, 1) = 1 + \frac{\rho^2}{2\log n}y^2 + \smallo{{(\log n)}^{-1}}.
\]
By using these expressions into Equation~\eqref{eq:  WT_Gaussian_model}, we have
\[\begin{split}
    &\Pro{Y_E > \zeta_1(n, y) \log n \mid X_E > \log n} \\
      &\simn \frac{\exp\left\lbrace\frac{2\rho^2 - \rho\left( \sqrt{\zeta_1} + 1/\sqrt{\zeta_1}\right)}{2(\omr)}\lln +\frac{1 - \rho/\sqrt{\zeta_1}}{2(\omr) }\log\zeta_1+ (1 - \kappa(\zeta_1,1))\log n\right\rbrace}{{(\omr)}^{-3/2}(\sqrt{\zeta_1}-\rho)(1 - \rho\sqrt{\zeta_1})}.
    \end{split}
\]
Since
\[
\begin{split}
    \left( 1 - \frac{\rho}{\sqrt{\zeta_1}}\right)\log \zeta_1 \simn& \frac{\log \rho}{\rho}\sqln y + \frac{1}{2 \rho^2} \frac{2\left(1-\rho^2\right)}{\log n} y^2\\
    & - \frac{\log \rho}{2} \llln-\frac{1-\rho^2}{2} \log \rho \frac{y^2}{\log n}+ \smallo{{(\log n)}^{-1}}\\
    =& \frac{\log \rho}{\rho}\sqln y - \frac {\log\rho}2 \llln\\
     &+ \left(\log\rho - \frac 2{\rho^2}\right)\frac{\omr}{2\log n}y^2+ \smallo{{(\log n)}^{-1}},\\
    \end{split}
    \]
    \[
\begin{split}
    (\sqrt{\zeta_1}-\rho)(1 - \rho\sqrt{\zeta_1})  \simn& \frac{{(\omr)}^{3/2}}{\sqrt{2log n}}y - \rho\frac{\omr}{4}\llln\\
    &- \rho\frac{1 - \rho^4}{4\log n}y^2+ \smallo{{(\log n)}^{-1}}
\end{split}
\]
and
\[    \begin{split}
    \frac{2\rho^2 - \rho\left( \sqrt{\zeta_1} + 1/\sqrt{\zeta_1}\right)}{2(\omr)}\simn& -\frac{1}{2} + \frac{1}{4\rho}\sqln y- \frac{1}{8}\llln\\
    &- \frac{1 - \rho^2}{8\log n}y^2+ \smallo{{(\log n)}^{-1}},
\end{split}
    \]
where we neglected all the terms characterized by a faster decay with respect to ${(\log n)}^{-1}$, we obtain
\[
    \Pro{Y_E > \zeta_1(n, y) \log n \mid X_E > \log n} \simn \frac{1}{y\sqrt{2\pi}}\e^{-{(\rho y)}^2/2}.
\]
\end{proof}
\begin{proof}[Proof of Corollary~\ref{coro: tilde_constants_Gaussian}]
Notice that
\[
    \Pro{Y_E > \zeta_1(n, y) \log n \mid X_E > \log n} \simn \frac{\rho}{\rho y\sqrt{2\pi}}\e^{-{(\rho y)}^2/2}=\rho\frac{\varphi(\rho y)}{\rho y},
\]
where $\varphi$ denotes the probability density function of the standard normal distribution.
Under the assumption that $y$ is large, it is straightforward to show, by using the Mill's ratio, that
\[
    \Pro{Y_E > \zeta_1(n, y) \log n \mid X_E > \log n} \simn \rho\left(1 - \Phi(\rho y)\right),
\]
which completes the proof.
\end{proof}

%%%%%%%%%%%%%%%%%%%%%%%%%%%%%%%%%%%%%%
\section{Appendix - Discussion on $k$}\label{app: other k}
%%%%%%%%%%%%%%%%%%%%%%%%%%%%%%%%%%%%%%

The present section contains some graphical results of the examples considered in Section \ref{sec: Examples}. We recall that each example runs a simulation of $5000$ replicates of $100000$ bivariate sample $(X,Y)$. 
As a result, we obtain a bivariate sample for $(V_1,V_2)$ of dimension $5000$ (each replicate of the $(X,Y)$ sample gives one couple of $(V_1,V_2)$), for each $k$. The values of $k$ indicating the dimension of the extreme set are fixed to $1, 50, 100, 250, 500, 1000, 1500, 2000$.
However, the figures do not show the results for each of the considered values: This choice is made to foster the readability, as a lot of almost similar plots on a single page may be redundant.
From a quick look at Figures~\ref{fig: SC surfaces} to~\ref{fig: Logistic marginals_V1_V2} for both examples, we can draw the same important conclusions:
\begin{enumerate}[label=(\alph*)]
    \item The range of the entries relative to $V_1$ is much wider than the range of the entries corresponding to $V_2$;
    \item The cloud of black points in the surfaces in Figure~\ref{fig: SC surfaces} and Figure~\ref{fig: Logistic surfaces} becomes closer and closer to the $v_1$ axis as $k$ increases.
\end{enumerate}
This is a clear effect of the positive dependence between $(X,Y)$.
To see this, think about the definition of $(V_1,V_2)$: $V_1$ is the maximum of the concomitants of the largest $k$ order statistics and $V_2$ is the maximum of the remaining set. The clear difference between the ranges of the two samples indicates that the highest concomitants contribute to the values of $V_1$ and the lowest concomitants are responsible of the relatively low values of $V_2$. Moreover, as $k$ increases, almost all the values that $V_2$ takes are lower than $1$, while $V_1$ can assume values which are greater than $100$, for example. Explained differently, high values of concomitants generally correspond to high order statistics of $X$.

It has to be said that the plotting range of the surfaces is chosen in order to enhance the visualization purposes, as extremely wide intervals would have been not suitable to observe the results.
However, there can be points which fall outside the plotting area: Figure~\ref{fig: SC surface k_1000} contains more points in the middle-right part, while Figure~\ref{fig: SC surface k_2000} contains more points in the upper-right part.
This indicates that when $k$ is $1000$, some concomitants high in value do not belong to the extreme set, but when $k$ increases to $2000$, the partition to which they belong changes. In other words, $V_1$ "steals" high values to $V_2$ when $k$ increases.

Plots relative to the marginal distributions support the presented conclusions: There is no noticeable difference among the log cdf shown in Figure~\ref{fig: SC marginals_V1} and Figure~\ref{fig: Logistic marginals_V1}, which means that the sample in $V_1$ does not change so much, while the upper end point of the distribution of $V_2$ appears to shrink towards $0$ (Figure~\ref{fig: SC marginals_V2} and Figure~\ref{fig: Logistic marginals_V2}). This is again an indication of the "stealing" argument: Since high values of concomitants contributing to $V_2$ pass into the extreme set as $k$ increases, only low values of concomitants are left to compute $V_2$, which becomes smaller and smaller. As previously said, this is the effect of the positive dependence between $(X,Y)$.

Finally, the extremely low errors shown in the boxplots in Figure~\ref{subfig: SC boxplots} and Figure~\ref{subfig: Logistic boxplots} ensure that the asymptotic approximation is almost perfect. The same conclusion can be made from the plots sowing the marginal errors in Figure~\ref{fig: SC marginals_V1_V2} and Figure~\ref{fig: Logistic marginals_V1_V2}. Despite some variability for low values of $v_1$ and $v_2$ (which is small anyway), the errors rapidly approach $0$ as $v_1$ and $v_2$ increase, which is important for practical applications as we are usually interested in studying what happens over high thresholds. 
\begin{figure}[h!]
    \centering
    \subfloat[][\emph{$k=1$.}]
    {\includegraphics[trim={5cm 1cm 4cm 2cm}, clip, width=.32\linewidth]{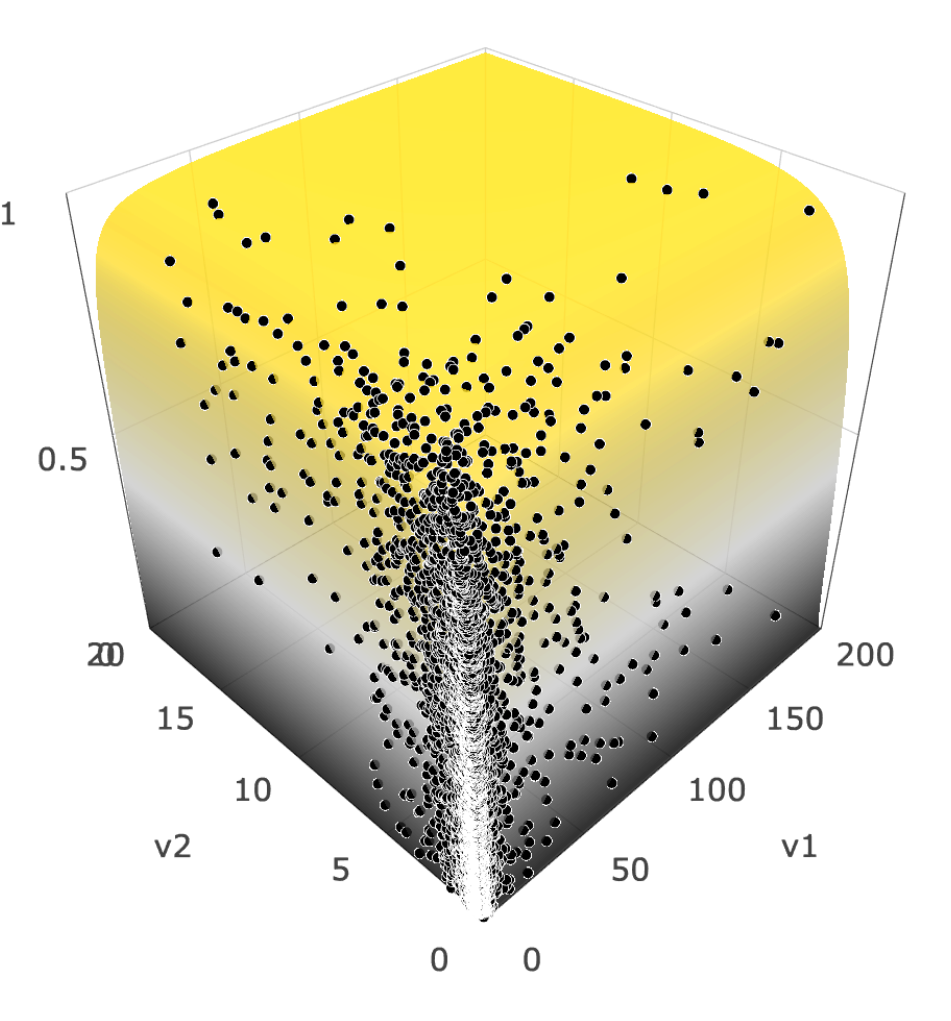}} \\
    \subfloat[][\emph{$k=50$.}]
    {\includegraphics[trim={5cm 1cm 4cm 2cm}, clip, width=.32\linewidth]{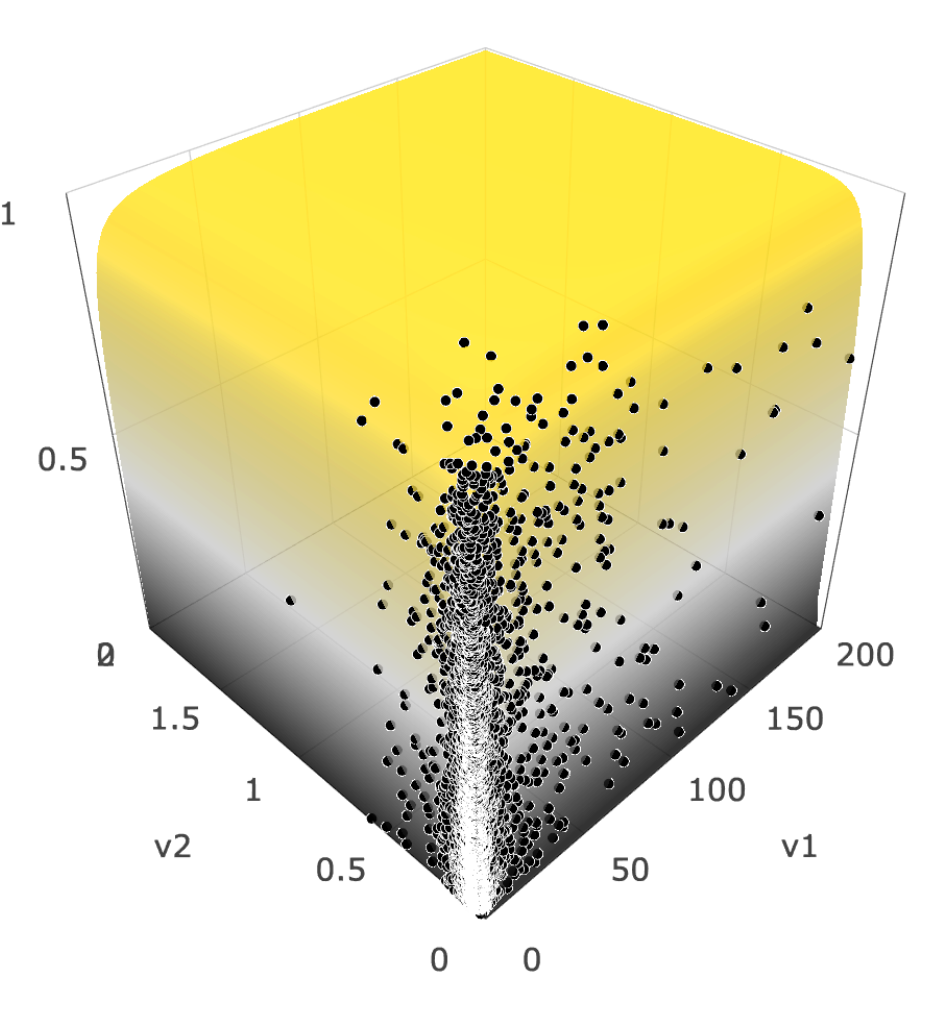}}
    \subfloat[][\emph{$k=100$.}]
    {\includegraphics[trim={5cm 1cm 4cm 2cm}, clip, width=.32\linewidth]{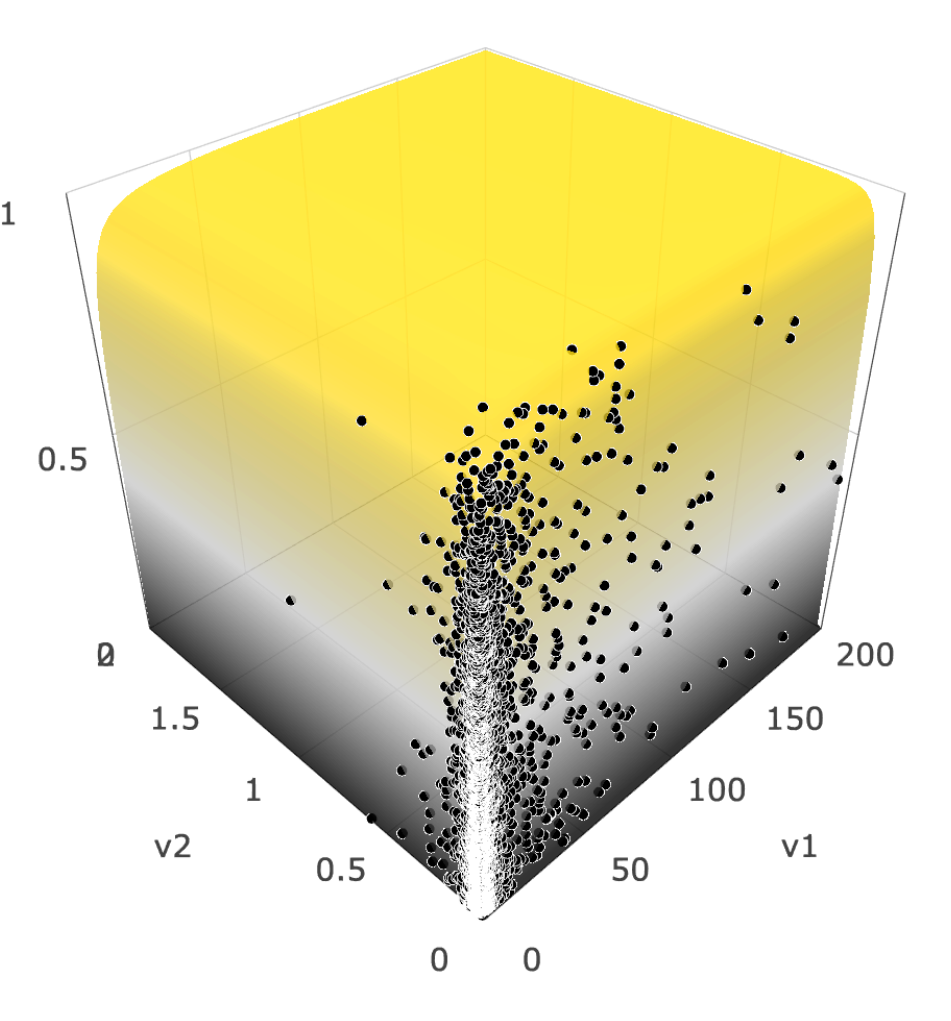}}\\
    \subfloat[][\emph{$k=250$.}]
    {\includegraphics[trim={5cm 1cm 4cm 2cm}, clip, width=.32\linewidth]{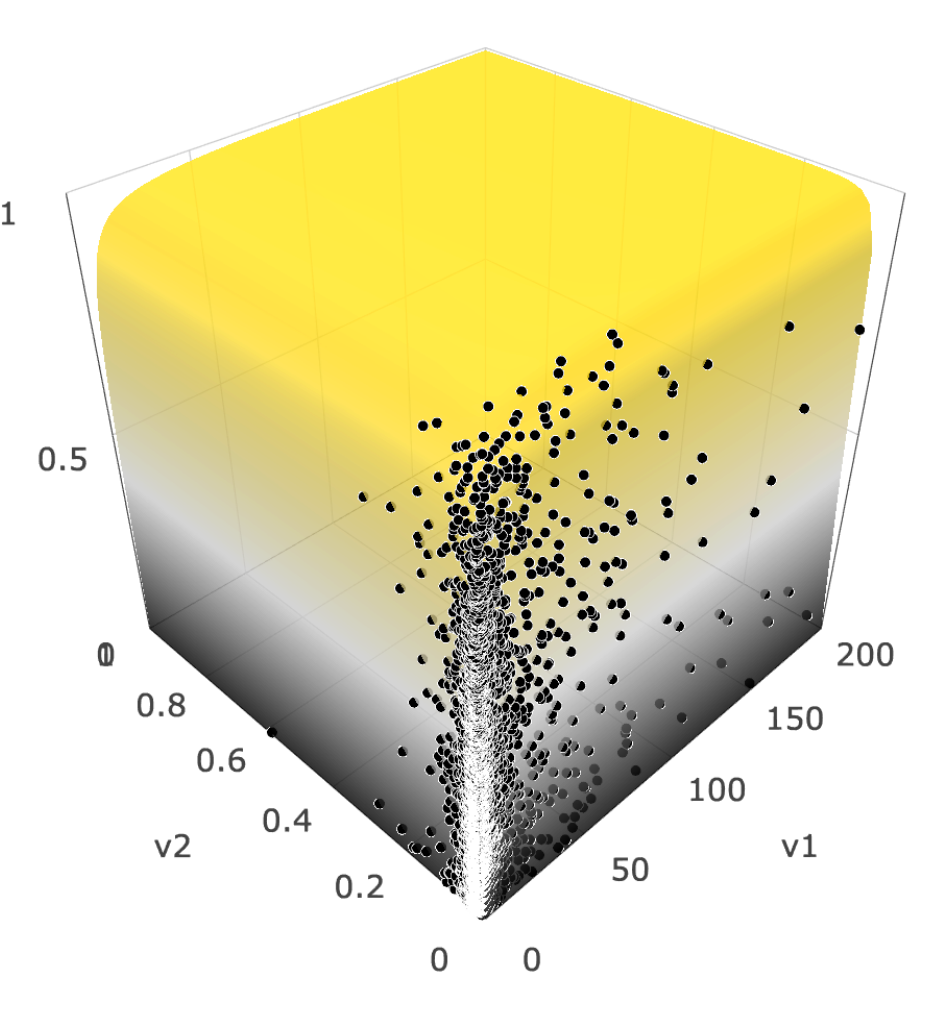}}
    % \subfloat[][\emph{$k=500$.}]
    % {\includegraphics[trim={5cm 1cm 4cm 2cm}, clip, width=.32\linewidth]{SC_Bivariate_k_500}}
    \subfloat[][\emph{$k=1000$.}\label{fig: SC surface k_1000}]
    {\includegraphics[trim={5cm 1cm 4cm 2cm}, clip,width=.32\linewidth]{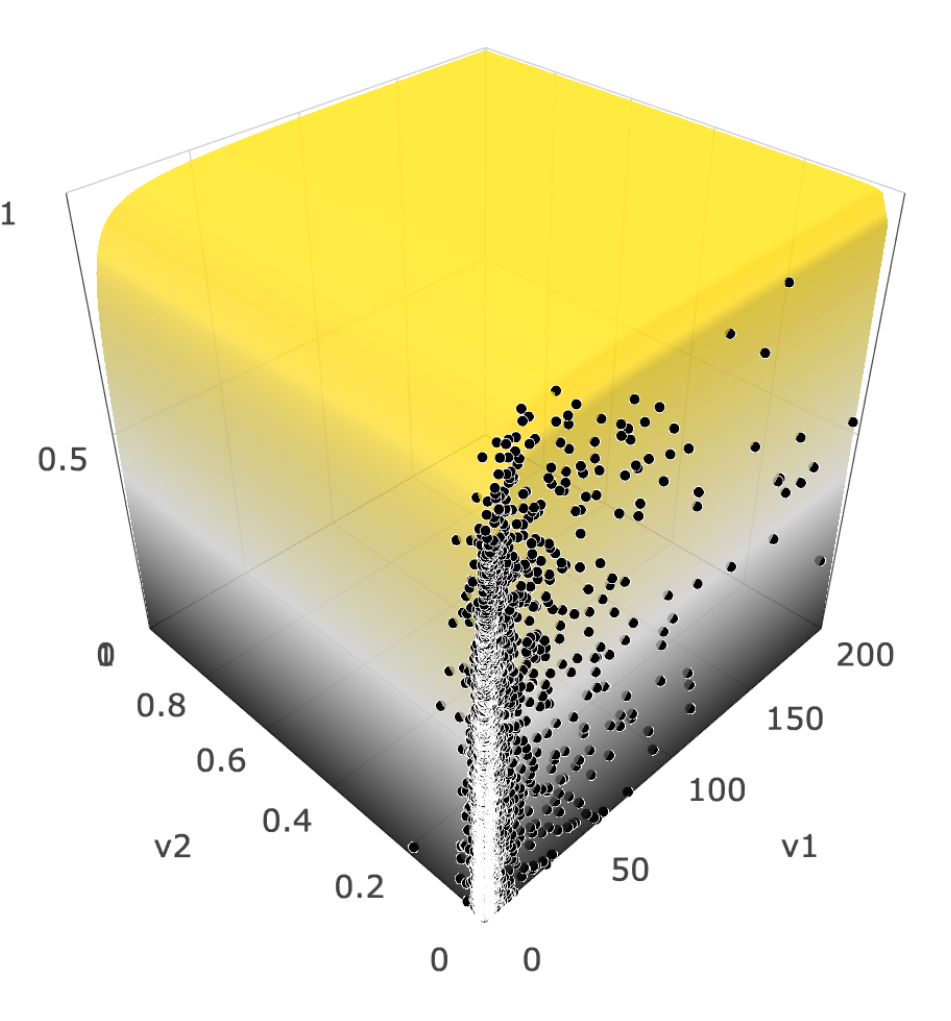}}
    % \subfloat[][\emph{$k=1500$.}]
    % {\includegraphics[trim={5cm 1cm 4cm 2cm}, clip, width=.32\linewidth]{SC_Bivariate_k_1500}}
    \subfloat[][\emph{$k=2000$.}\label{fig: SC surface k_2000}]
    {\includegraphics[trim={5cm 1cm 4cm 2cm}, clip,width=.32\linewidth]{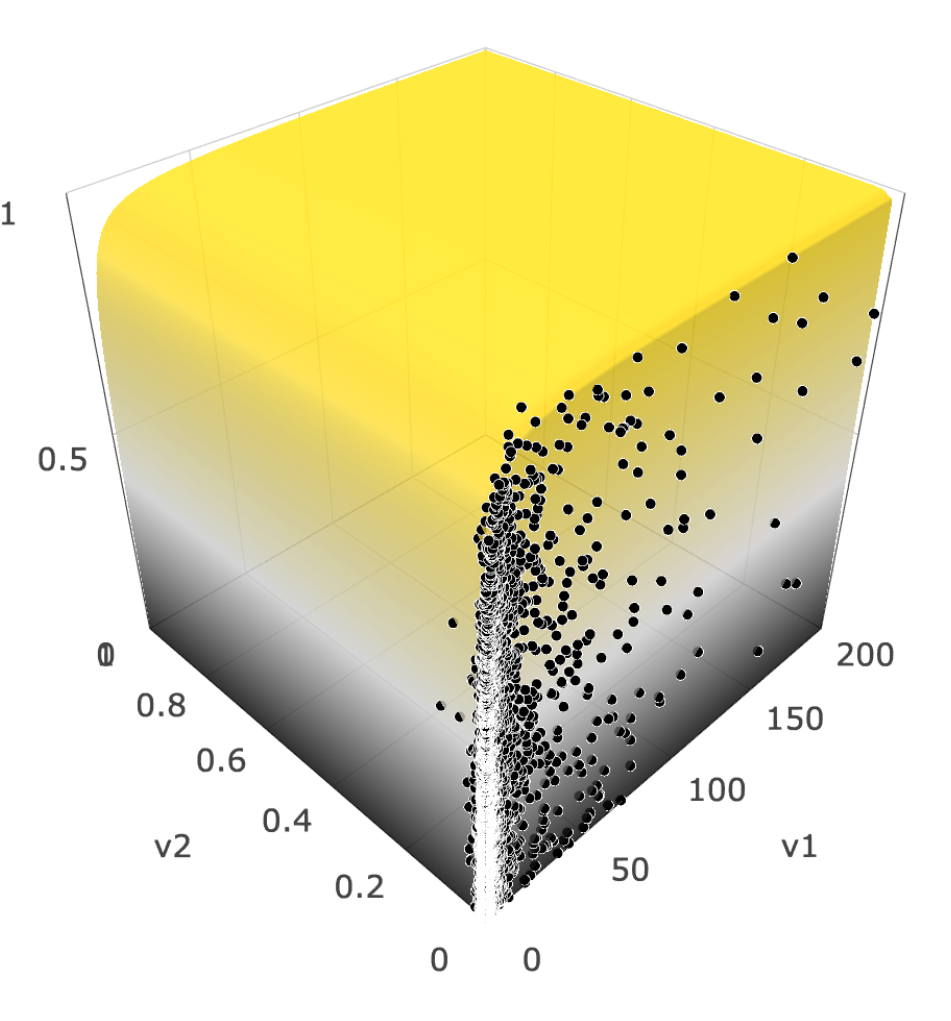}}\\
    \subfloat[][\emph{Boxplots of the errors.}\label{subfig: SC boxplots}]
    {\includegraphics[trim={0 6cm 0 6cm}, clip,width=\linewidth]{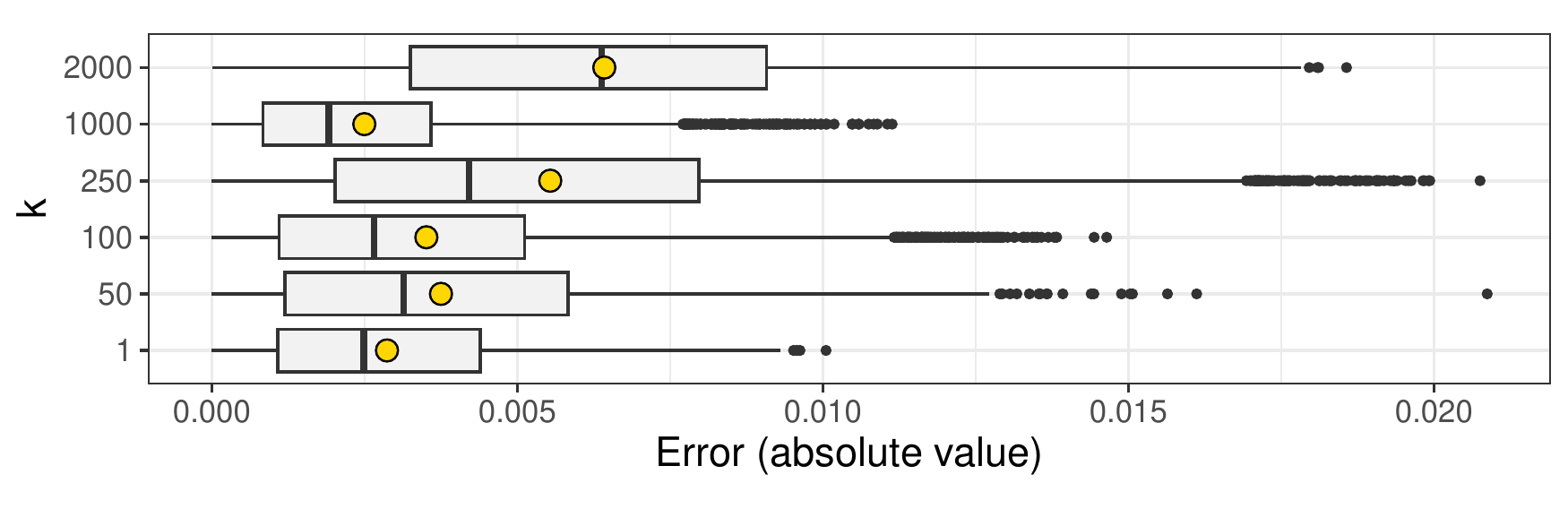}}
    \caption{Survival Clayton$(\theta=2)$. Bivariate cdf~computed by numerically integrating equation \eqref{eq: Joint_V1V2}, for various values of $k$. The black bullets correspond to the bivariate empirical cdf of the simulated samples. Panel g) shows the boxplots of the errors in absolute value between the bivariate asymptotic cdf computed in $(V_{1,i},V_{2,i})$ for $i=1,\dots,5000$ and the corresponding empirical bivariate cdf. The gold dots represent the error means.}\label{fig: SC surfaces}
\end{figure}
\begin{figure}[h!]
    \centering
    \subfloat[][\emph{$V_1$: log marginal cdf comparison and marginal error behaviour with respect to $v_1$.}\label{fig: SC marginals_V1}]
    {\includegraphics[trim={0cm 4cm 0cm 1cm}, clip,width=\linewidth]{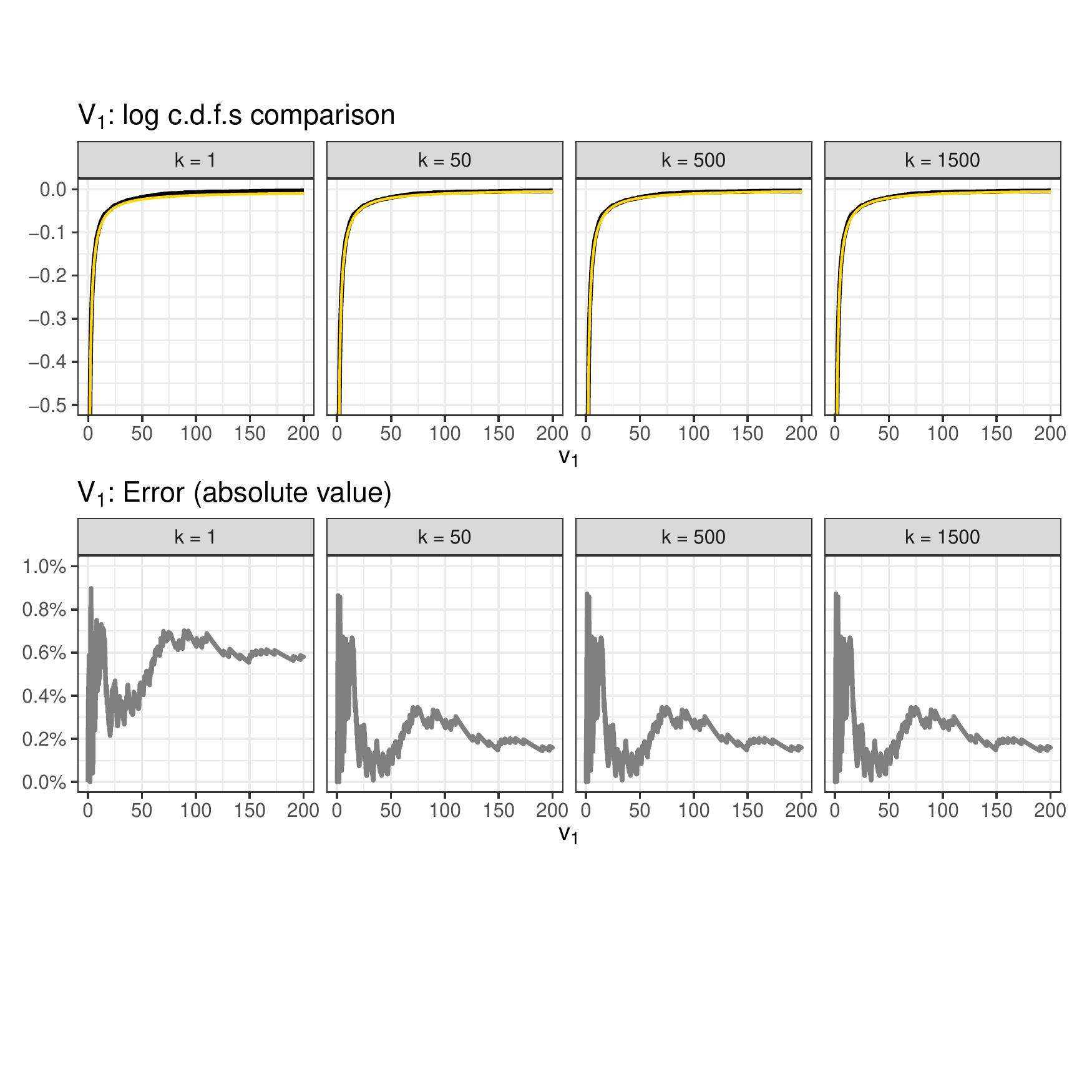}} \\
    \subfloat[][\emph{$V_2$: log marginal cdf comparison and marginal error behaviour with respect to $v_2$.}\label{fig: SC marginals_V2}]
    {\includegraphics[trim={0cm 4cm 0cm 1cm}, clip,width=\linewidth]{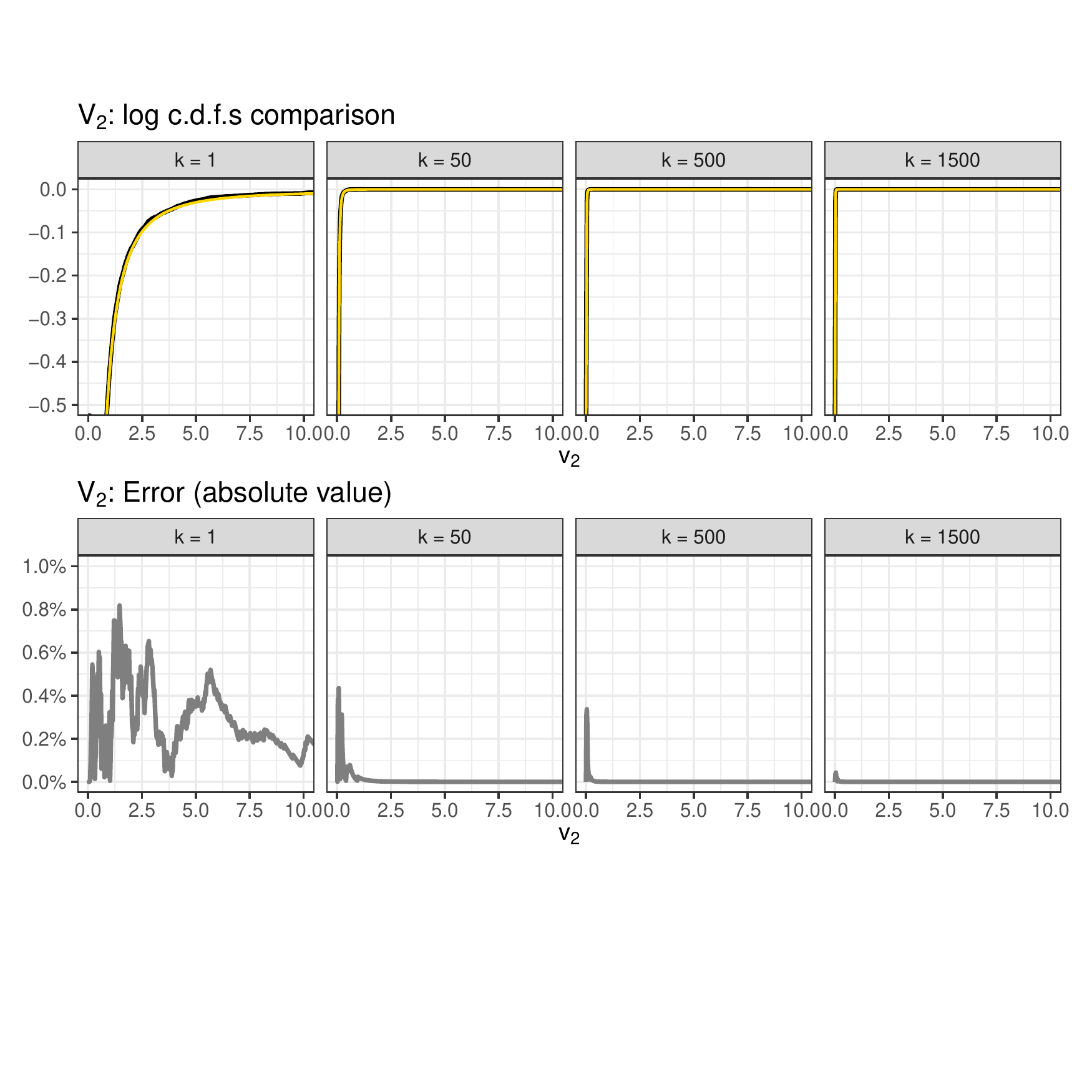}}

    \caption{Survival Clayton$(\theta=2)$. Empirical univariate cdf of $V_1$ and $V_2$, together with the asymptotic theoretical values obtained by numerically  integrating the asymptotic joint cdf~\eqref{eq: Joint_V1V2}.}\label{fig: SC marginals_V1_V2}
\end{figure}

%\paragraph{Logistic}-----------------------------------
\begin{figure}[h]
    \centering
    \subfloat[][\emph{$k=1$.}]
    {\includegraphics[trim={5cm 1cm 4cm 2cm}, clip, width=.32\linewidth]{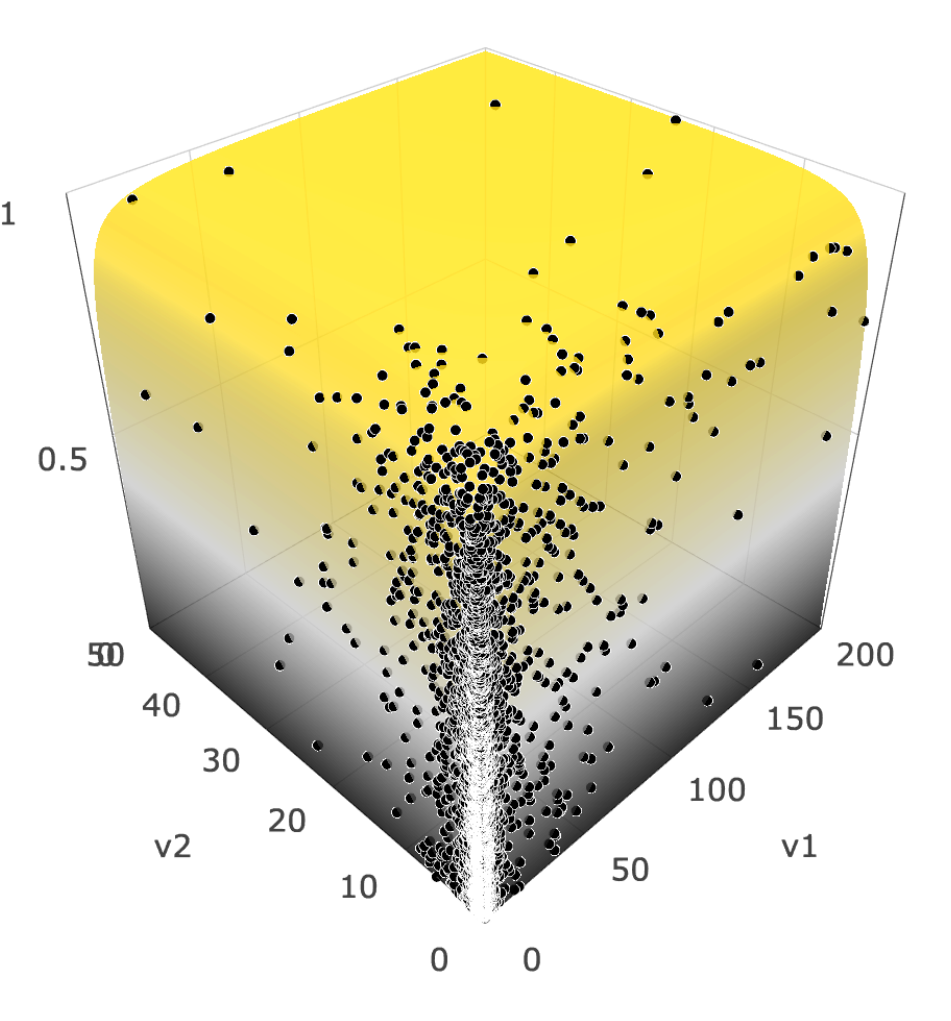}} \\
    \subfloat[][\emph{$k=50$.}]
    {\includegraphics[trim={5cm 1cm 4cm 2cm}, clip, width=.32\linewidth]{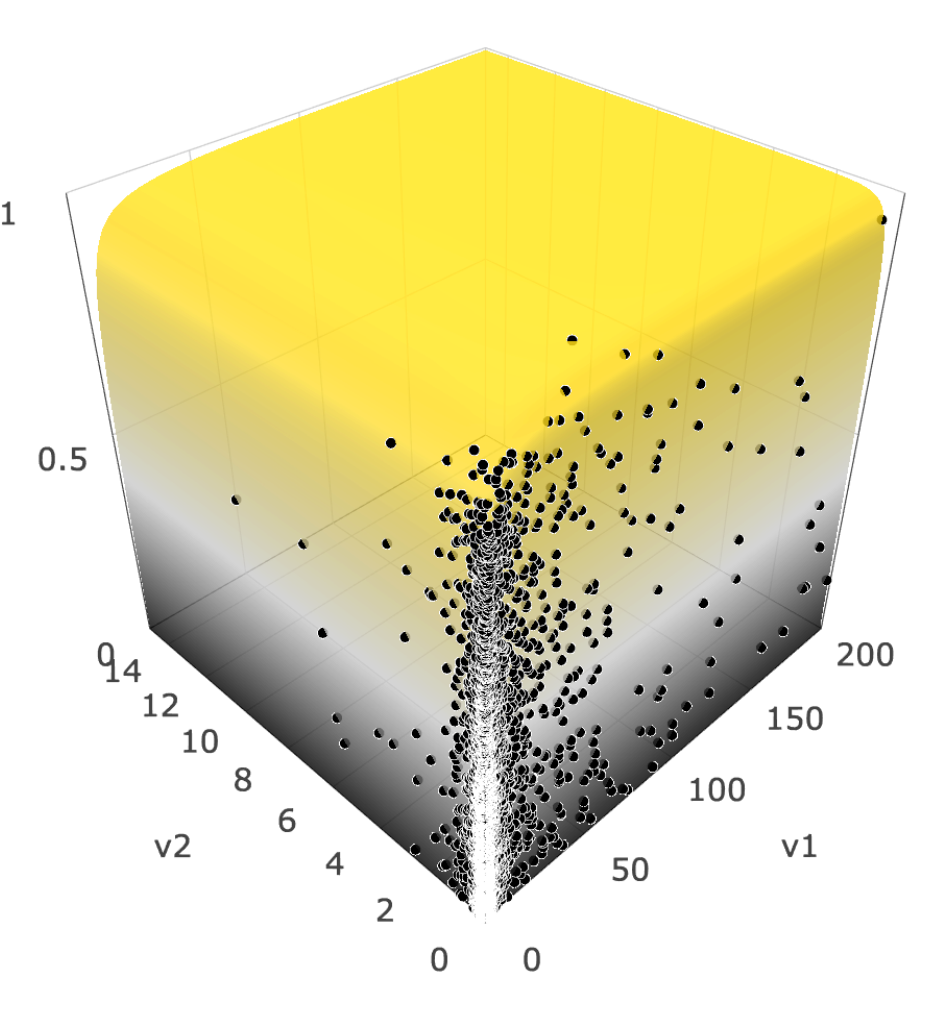}}
    \subfloat[][\emph{$k=100$.}]
    {\includegraphics[trim={5cm 1cm 4cm 2cm}, clip, width=.32\linewidth]{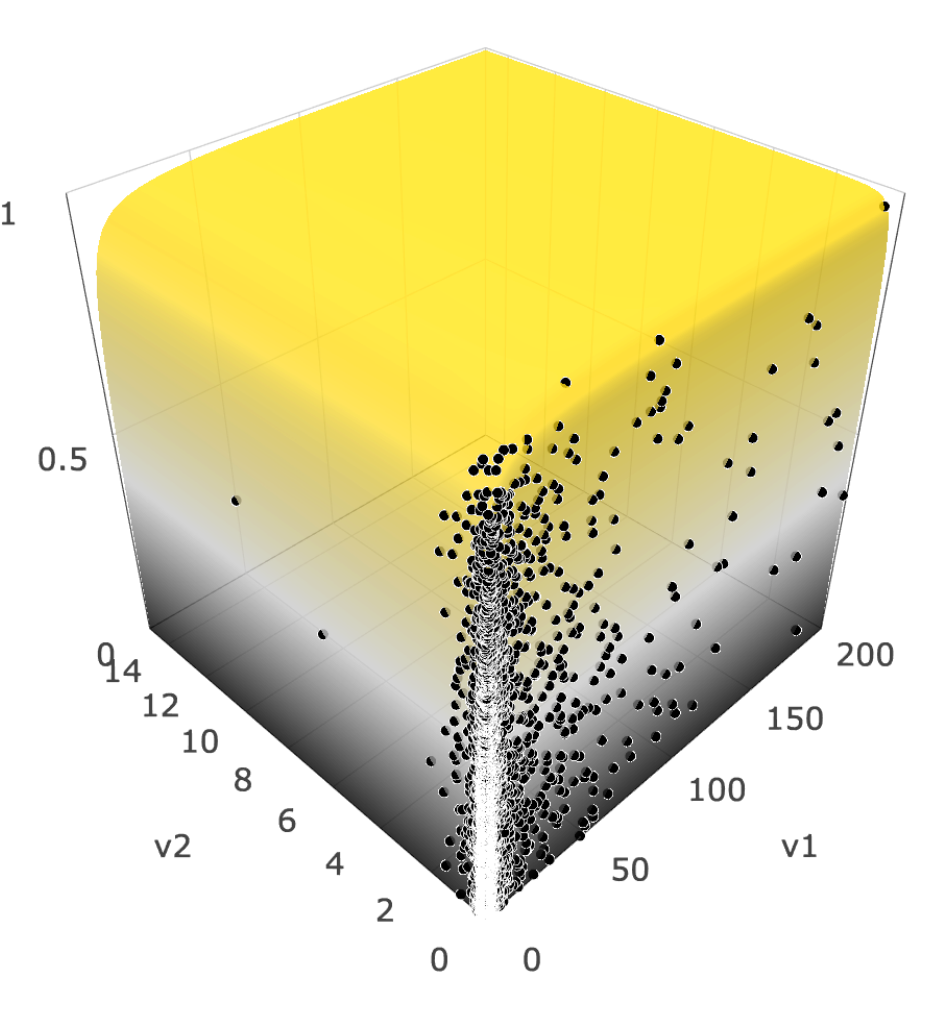}}\\
    \subfloat[][\emph{$k=250$.}]
    {\includegraphics[trim={5cm 1cm 4cm 2cm}, clip, width=.32\linewidth]{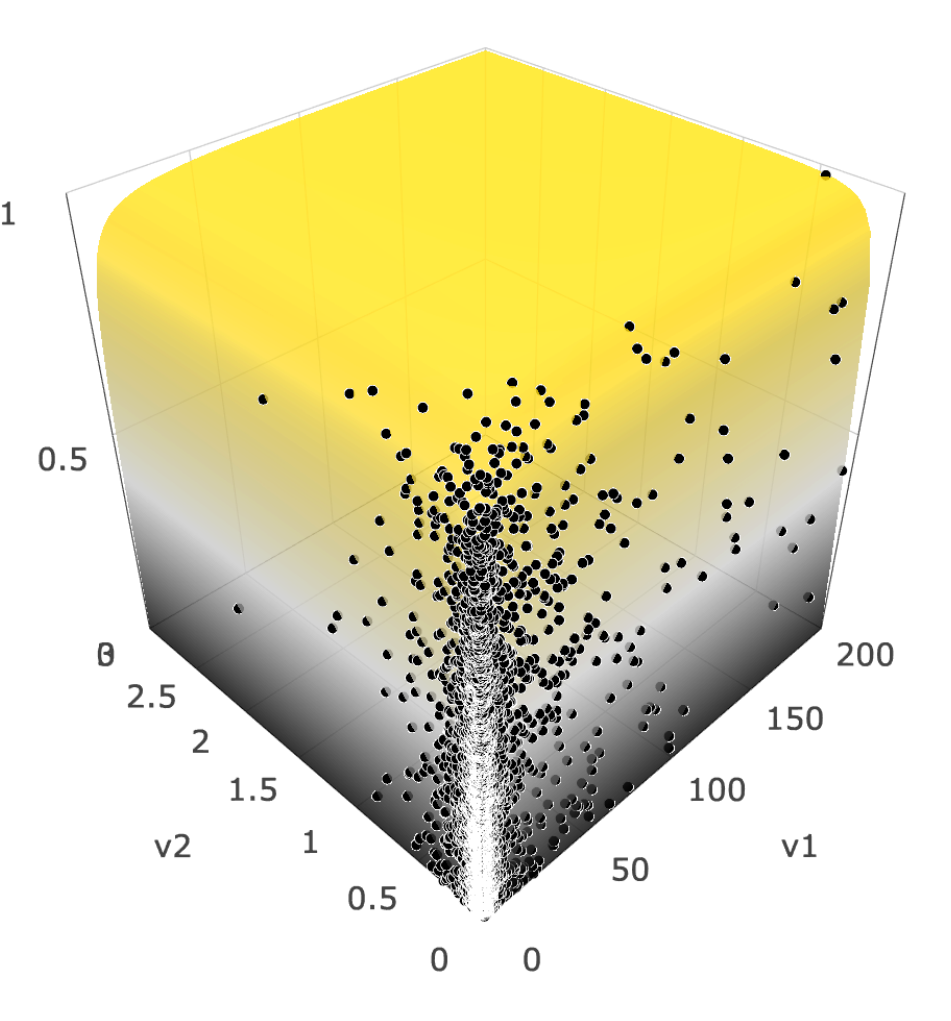}}
    % \subfloat[][\emph{$k=500$.}]
    % {\includegraphics[trim={5cm 1cm 4cm 2cm}, clip, width=.32\linewidth]{Logistic_Bivariate_k_500}}
    \subfloat[][\emph{$k=1000$.}]
    {\includegraphics[trim={5cm 1cm 4cm 2cm}, clip,width=.32\linewidth]{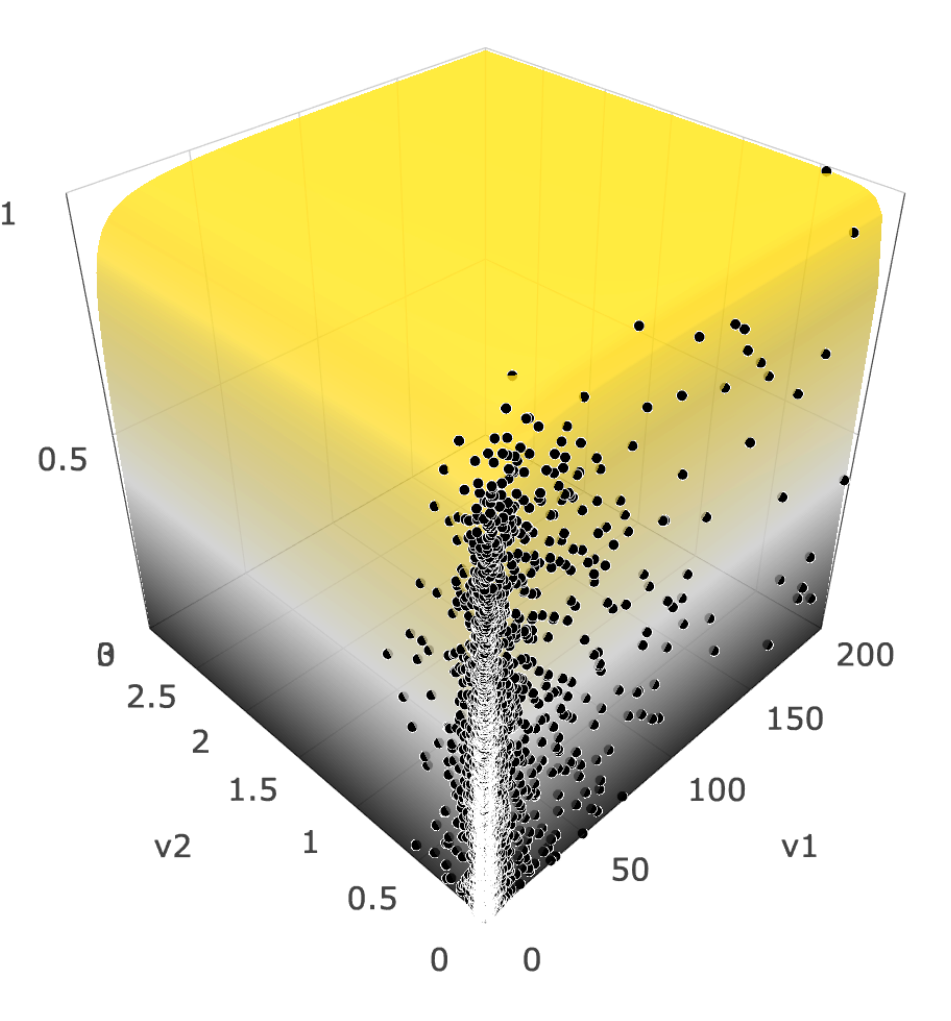}}
    % \subfloat[][\emph{$k=1500$.}]
    % {\includegraphics[trim={5cm 1cm 4cm 2cm}, clip, width=.32\linewidth]{Logistic_Bivariate_k_1500}}
    \subfloat[][\emph{$k=2000$.}]
    {\includegraphics[trim={5cm 1cm 4cm 2cm}, clip,width=.32\linewidth]{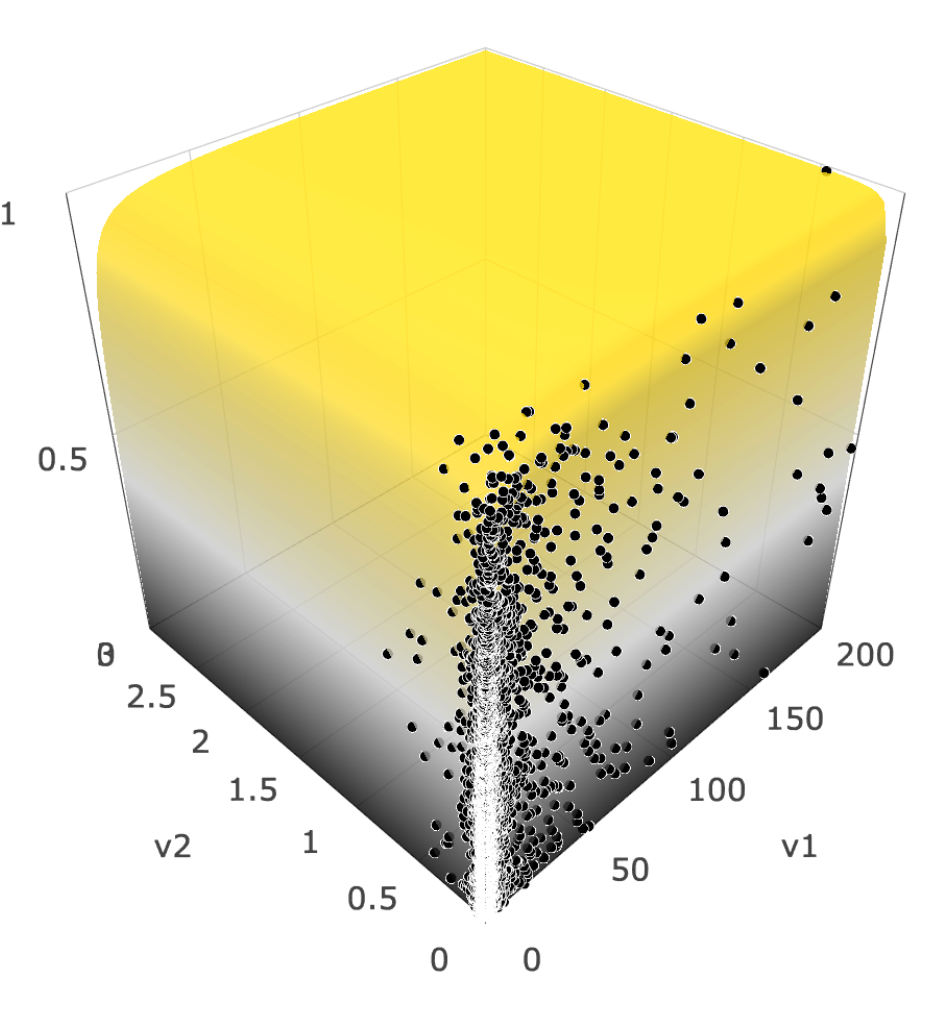}}\\
    \subfloat[][\emph{Boxplots of the errors.}\label{subfig: Logistic boxplots}]
    {\includegraphics[trim={0 6cm 0 6cm}, clip,width=\linewidth]{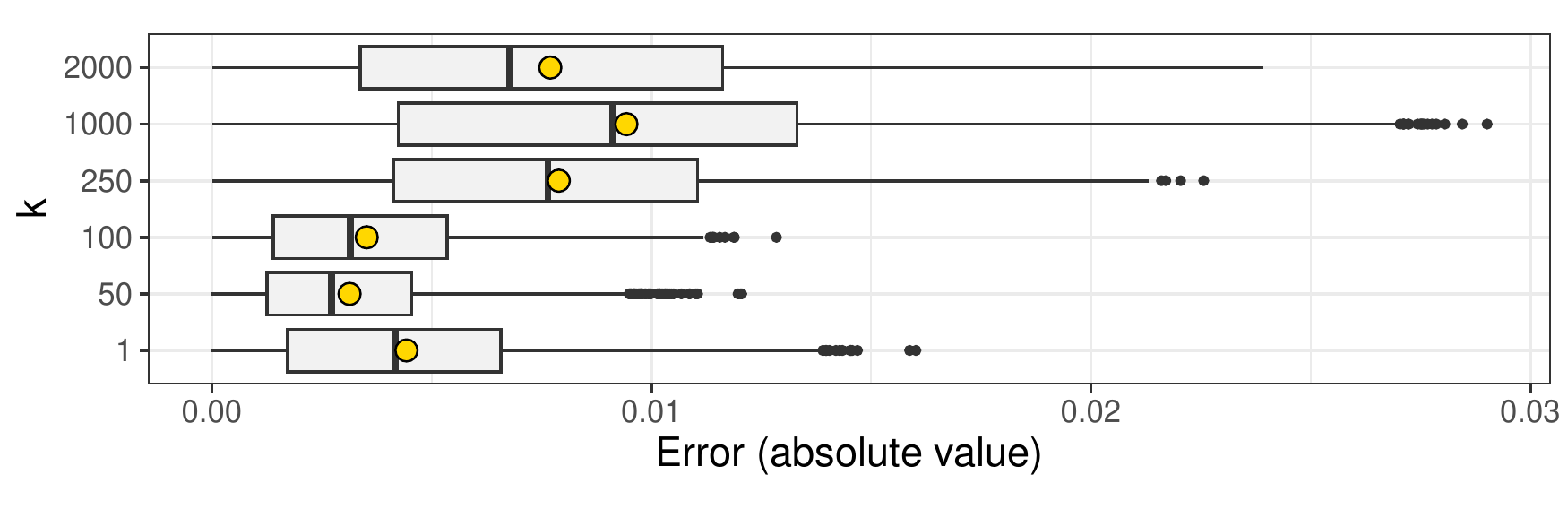}}
    \caption{Logistic$(\gamma = 1/2)$. Bivariate cdf~computed by numerically integrating equation \eqref{eq: Joint_V1V2}, for various values of $k$. The black bullets correspond to the bivariate empirical cdf of the simulated samples. Panel g) shows the boxplots of the errors in absolute value between the bivariate asymptotic cdf computed in $(V_{1,i},V_{2,i})$ for $i=1,\dots,5000$ and the corresponding empirical bivariate cdf. The gold dots represent the error means.}\label{fig: Logistic surfaces}
\end{figure}
\begin{figure}[h]
    \centering
    \subfloat[][\emph{$V_1$: log marginal cdf comparison and marginal error behaviour with respect to $v_1$.}\label{fig: Logistic marginals_V1}]
    {\includegraphics[trim={0cm 4cm 0cm 1cm}, clip,width=\linewidth]{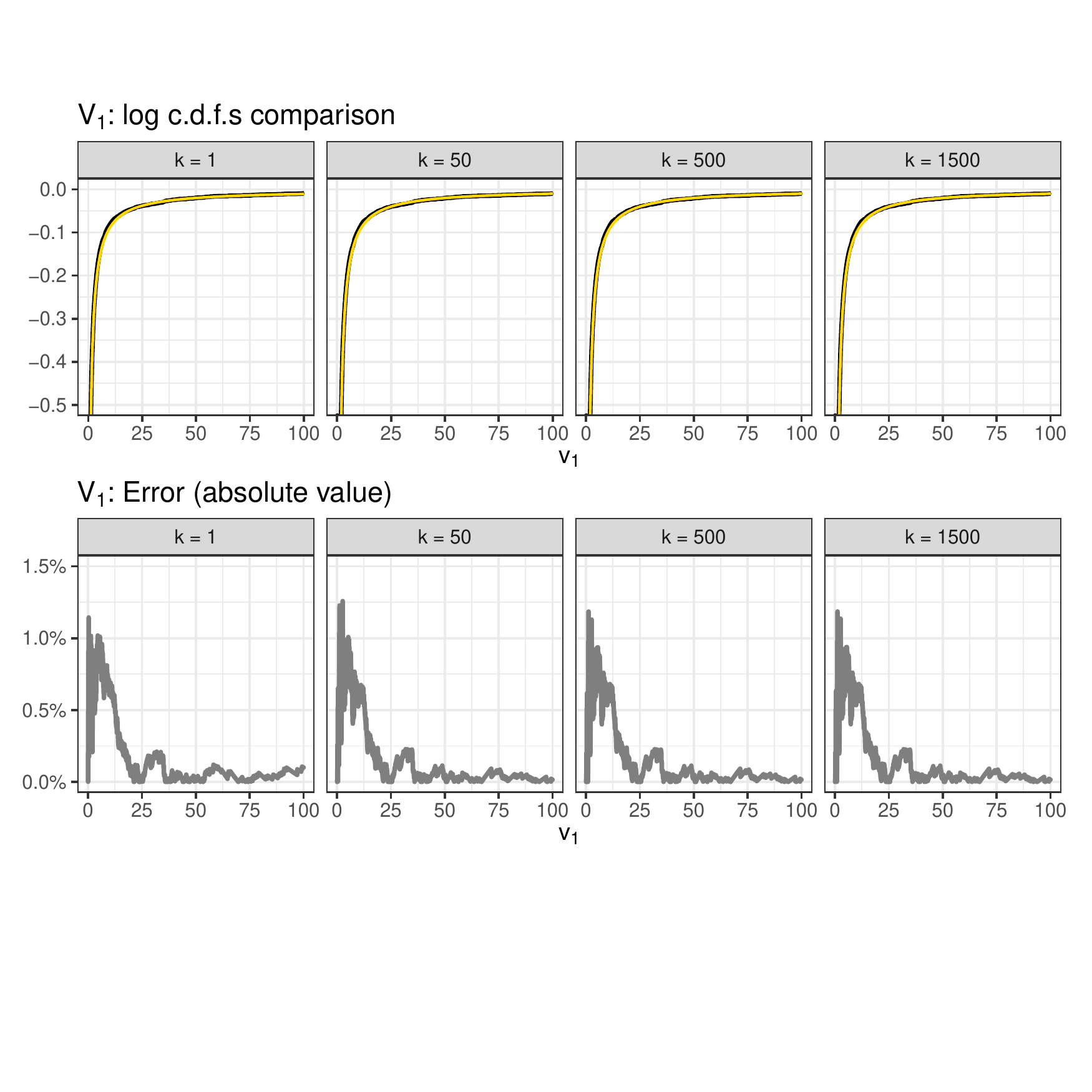}} \\
    \subfloat[][\emph{$V_2$: log marginal cdf comparison and marginal error behaviour with respect to $v_2$.}\label{fig: Logistic marginals_V2}]
    {\includegraphics[trim={0cm 4cm 0cm 1cm}, clip,width=\linewidth]{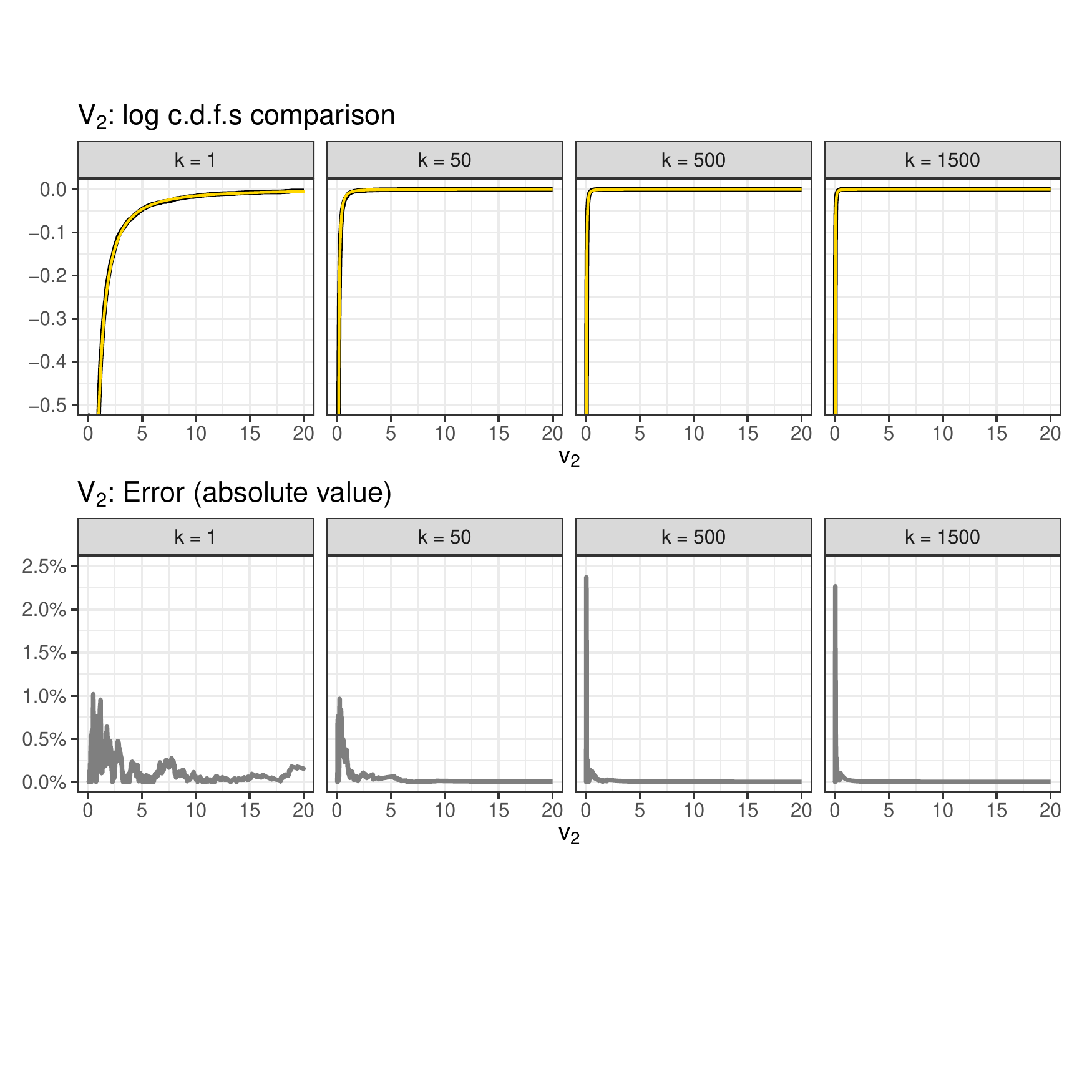}}

    \caption{Logistic$(\gamma = 1/2)$ example. Empirical univariate cdf of $V_1$ and $V_2$, together with the asymptotic theoretical values obtained by numerically  integrating the asymptotic joint cdf~\eqref{eq: Joint_V1V2}.}\label{fig: Logistic marginals_V1_V2}
\end{figure}
\clearpage

\section{Appendix - The Von Mises conditions}\label{app:def}
%----------------------------------
This section reports the Von Mises conditions, as written in \cite{Resnick1987}, borrowing his notation.
% The convenience of these conditions is the fact that they are easier to verify with respect to others presented for example in \cite{Resnick1987}.
%
\\[1ex]

\textit{Proposition 1.15 (\cite{Resnick1987}): %\label{prop: Von_Mises_15}
    Suppose $F$ is absolutely continuous with positive density $f$ in some neighborhood of $\infty$.
    \begin{enumerate}[(\alph*)]
        \item If for some $\alpha>0$
        \begin{equation}\label{eq: Resnick_1.19}
            \lim_{x \rightarrow \infty} x f(x) /(1-F(x))=\alpha
        \end{equation}
        then $F \in \mathscr{D}\left(\Phi_\alpha\right)$. We may choose $a_n$ to satisfy $a_n f\left(a_n\right) \sim \alpha / n$.
        \item If $f$ is nonincreasing and $F \in \mathscr{D}\left(\Phi_\alpha\right)$ then \eqref{eq: Resnick_1.19} holds.
        \item Equation \eqref{eq: Resnick_1.19} holds iff for some $z_0$ and all $x>z_0$, we have
        \[
        1-F(x)=c \exp \left\{-\int_{z_0}^x t^{-1} \alpha(t) \dd{t}\right\}
        \]
        where $\lim _{t \rightarrow \infty} \alpha(t)=\alpha$ and $c$ is a positive constant.\\[1ex]
    \end{enumerate}
}%
\textit{Proposition 1.16 (\cite{Resnick1987}): %}]\label{prop: Von_Mises_16}
Suppose $F$ has finite right end-point $x_0$ and is absolutely continuous in a left neighborhood of $x_0$ with positive density $f$.
    \begin{enumerate}[(\alph*)]
        \item If for some $\alpha>0$
        \begin{equation}\label{eq: Resnick_1.20}
            \lim _{x \uparrow x_0}\left(x_0-x\right) f(x) /(1-F(x))=\alpha
        \end{equation}
        then $F \in \mathscr{D}\left(\Psi_\alpha\right)$.
        \item If $f$ is nonincreasing and $F \in \mathscr{D}\left(\Psi_\alpha\right)$ then \eqref{eq: Resnick_1.20} holds.
        \item Equation \eqref{eq: Resnick_1.20} holds iff $c(x)$ can be taken to be constant in some left neighborhood of $x_0$ in the representation
        \[
                1-F(x)=c(x) \exp \left\{-\int_{x_0-1}^x \delta(t) /\left(x_0-t\right)\dd{t} \right\}\quad \text { and for } x<x_0 
        \]
        where $\lim_{t \uparrow x_0} \delta(t)=\alpha, \, \lim_{t \uparrow x_0} c(t)=c_0>0$.\\[1ex]
    \end{enumerate}
}%
\textit{Proposition 1.17 (\cite{Resnick1987}): %}]\label{prop: Von_Mises_17}
Let $F$ be absolutely continuous in a left neighborhood of $x_0$ with density $f$.
    \begin{enumerate}[(\alph*)]
    \item If
        \begin{equation}\label{eq: Resnick_1.21}
            \lim _{t \uparrow x_0} f(x) \int_x^{x_0}(1-F(t)) \dd{t} /(1-F(x))^2=1
        \end{equation}
        then $F \in D(\Lambda)$. In this case we may take
        \begin{align*}
        f(t) & =\int_t^{x_0}(1-F(s)) \dd{s} /(1-F(x)) \\
        b_n & =(1 /(1-F))^{\leftarrow}(n), \quad a_n=f\left(b_n\right) .
        \end{align*}
    \item If $f$ is nonincreasing and $F \in \mathscr{D}(\Lambda)$ then  holds.
    \item Equation \eqref{eq: Resnick_1.21} holds  iff
        \begin{equation}\label{eq: Resnick_1.23}
        1-F(x)=c \exp \left\{-\int_{z_0}^x(g(t) / f(t)) \dd{t}\right\}, \quad z_0<x<x_0,
        \end{equation}
        where $\lim_{t \uparrow_0} g(x)=1$ and $f$ is absolutely continuous with density $f(x) \rightarrow 0$ as $x \uparrow x_0$.
    \item Equation \eqref{eq: Resnick_1.21} or \eqref{eq: Resnick_1.23} are equivalent to $t f\left((1 /(1-F))^{\leftarrow}(t)\right) \in {RV}_0$ (\cite{Sweeting1985}).
    \end{enumerate}
}
\end{appendix}
%%%%%%%%%%%%%%%%%%%%%%%%%%%%%%%%%%%%%%%%%%%%%%
%% Support information, if any,             %%
%% should be provided in the                %%
%% Acknowledgements section.                %%
%%%%%%%%%%%%%%%%%%%%%%%%%%%%%%%%%%%%%%%%%%%%%%
\begin{acks}[Acknowledgments]
This paper was mainly written during the first author's visiting period at the ESSEC Business School, Paris, France and during the second author's visiting period at the ESOMAS Department, University of Turin, Turin, Italy. Both authors are grateful to  the hosting institutions.
\end{acks}
%%%%%%%%%%%%%%%%%%%%%%%%%%%%%%%%%%%%%%%%%%%%%%
%% Funding information, if any,             %%
%% should be provided in the                %%
%% funding section.                         %%
%%%%%%%%%%%%%%%%%%%%%%%%%%%%%%%%%%%%%%%%%%%%%%
\begin{funding}
The authors are grateful for funding to Labex MME DII (ANR-11-LABX-0023-01) and ESOMAS department (University of Turin).
\end{funding}

%%%%%%%%%%%%%%%%%%%%%%%%%%%%%%%%%%%%%%%%%%%%%%%%%%%%%%%%%%%%%
\bibliographystyle{imsart-nameyear} % Style BST file (imsart-number.bst or imsart-nameyear.bst)
\bibliography{AAP_Dep_project.bib}       % Bibliography file (usually '*.bib')

\end{document}